\documentclass[10pt]{amsart}
\usepackage{amsmath}
\usepackage{amssymb}
\usepackage{amsthm}
\usepackage{amscd}
\usepackage{amsfonts}
\usepackage{graphicx}
\usepackage{fancyhdr}
\usepackage[utf8]{inputenc}
\usepackage{mathrsfs}
\usepackage{tikz-cd}
\usepackage{hyperref}
\usepackage[all]{xy}

\numberwithin{equation}{section}
\newtheorem{theorem}[equation]{Theorem}
\newtheorem{thm}{Theorem}
\theoremstyle{plain}
\newtheorem{lemma}[equation]{Lemma}
\newtheorem{proposition}[equation]{Proposition}
\newtheorem{definition}[equation]{Definition}
\newtheorem{corollary}[equation]{Corollary}

\newtheorem*{corollary*}{Corollary}

\newtheorem{remark}[equation]{Remark}

\def\GL{\mathrm{GL}}
\def\GSp{\mathrm{GSp}}

\def\Sp{\mathrm{Sp}}
\def\Stab{\mathrm{Stab}}
\def\GU{\mathrm{GU}}

\def\Nilp{\mathrm{Nilp}}

\def\det{\mathrm{det}}
\def\ord{\mathrm{ord}}
\def\Lie{\mathrm{Lie}}
\def\inv{\mathrm{inv}}

\def\Sets{\mathrm{Sets}}
\def\Ker{\mathrm{Ker}}

\def\Nrd{\mathrm{Nrd}}
\def\Trd{\mathrm{Trd}}

\DeclareMathOperator{\End}{End}
\DeclareMathOperator{\Adm}{Adm}

\def\calL{\mathcal{L}}
\def\calN{\mathcal{N}}
\def\calM{\mathcal{M}}
\def\calO{\mathcal{O}}

\def\gothS{\mathfrak{S}}

\def\CC{\mathbb{C}}
\def\DD{\mathbb{D}}

\def\FF{\mathbb{F}}
\def\GG{\mathbb{G}}

\def\PP{\mathbb{P}}
\def\QQ{\mathbb{Q}}
\def\RR{\mathbb{R}}

\def\XX{\mathbb{X}}
\def\ZZ{\mathbb{Z}}

\newcommand{\Iw}{\mathrm{Iw}}
\newcommand{\Spf}{\mathrm{Spf}}

\newcommand{\Dieu}{Dieudonn\'{e}}

\newcommand{\Sh}{\gothS h}

\newcommand{\der}{\mathrm{der}}

\usepackage[top=1.3in, bottom=1.3in, left=1.3in, right=1.3in]{geometry}
\author{Haining Wang}
\address{\parbox{\linewidth} {Haining Wang\\ Department of Mathematics,\\ McGill University,\\ 805 Sherbrooke St W,\\ Montreal, QC H3A 0B9, Canada.~ }}
\email{wanghaining1121@outlook.com}

\subjclass[2000]{Primary 11G18, Secondary 20G25}
\date{\today}

\begin{document}
\title{On quaternionic unitary Rapoport-Zink spaces with parahoric level structures}

\keywords{\emph{Shimura varieties, Rapoport-Zink spaes, affine Deligne-Lusztig varieties}}

\begin{abstract}
In this article, we describe the underlying reduced scheme of a quaternionic unitary Rapoport-Zink space with Iwahori level structure.  In a previous work, we have studied the quaternionic unitary Rapoport-Zink space with a special maximal parahoric level structure. We will describe the morphism between these two Rapoport-Zink spaces and the fiber of this morphism. As an application of this result, we describe the the quaternionic unitary Rapoport-Zink space with any parahoric level structure. 
\end{abstract}

\maketitle

\tableofcontents
\section{Introduction}
This article contributes to the problem of explicitly describing the supersingular locus in the reduction modulo an odd prime $p$ of an integral model of a Shimura variety. This work is a natural continuation of \cite{Wang-2} where we studied the Shimura variety and Rapoport-Zink space 
related to the group $\GU_{B}(2)$ where $B$ is an indefinite quaternion algebra over $\QQ$ with a maximal parahoric level structure. This parahoric level structure corresponds to the fact that the Shimura variety or the Rapoport-Zink space classifies abelian varieties or $p$-divisible groups with principal polarizations. In this article we are concerned with the Shimura variety and the Rapoport-Zink space for the same group but with the Iwahori level structure. As an application, we can extend our descriptions to arbitrary parahoric level structures. There are a few cases where one knows how to describe the supersingular locus of a Shimura variety with Iwahori level structure in an explicit way. We mention the following results. 

\begin{enumerate}
\item In \cite{Stam97}, the author treats the case corresponding to $\GL_{2}/F$ where $F$ is a real quadratic field. This is the case of a  Hilbert modular surface with Iwahori level structure.
\item In \cite{GY12}, the authors treat the case $\GSp(2g)$ for general $g$. This is the case of Siegel modular varieties with Iwahori level structure. In this case, the authors can only describe parts of the supersingular locus namely those Kottwitz-Rapoport strata that are completely contained in the supersingular locus. The case of $\GSp(4)$ with Iwahori level structure is special and the supersingular locus can be completely described. This is done in \cite{Yu08}. 

\item In \cite{GHN16}, the authors introduce the concept of a fully Hodge-Newton decomposable pair $(G, \mu)$ where $G$ is a quasi-simple reductive group over a non-archimedean local field and $\mu$ is a miniscule cocharacter of $G$ over the algebraic closure $\bar{F}$. If $(G, \mu)$ arises from a local Shimura datum, then one can expect to describe the reduced scheme of the corresponding Rapoport-Zink space with any parahoric level structure in terms of classical Deligne-Lusztig varieties. The quaternionic unitrary case studied in this article gives such a fully Hodge-Newton decomposable pair. 
\end{enumerate}

In \cite{Stam97}, it is shown that the supersingular locus of the Hilbert modular surface with Iwahori level structure is purely $2$-dimensional and therefore  consists of irreducible components of the whole Shimura variety. In fact one can classify all the Shimura varieties with parahoric level structure that a similar phenomenon occur. This result is obtained in the recent preprint of \cite{GHR19} using purely group theoretic method.  We will show below that the supersingular locus in our case is not quite equidimensional: in addition to irreducible components of dimension $3$, there are also irreducible components of dimension $2$. Although our results and methods don't rely directly on that of \cite{GHR19}, we are indeed inspired by their results and computations. Moreover we also include a section that compare our results with the predictions in \cite{GHN16} and \cite{GHR19} using purely group theoretic method. 

Finally we mention that the group $\GU_{B}(2)$ is a non-trivial inner form of the split symplectic group $\GSp(4)$. The relation between our Shimura variety and the classical Siegel threefold should resemble the relation between the Shimura curve associated to an indefinite quaternion algebra and the classical modular curve. Therefore our results should have many interesting arithmetic applications and we will report one application in \cite{Wang-pre}. 

\subsection{Results on Rapoport-Zink space} We describe our results in more details. Let $D$ be the quaternion division algebra over $\QQ_{p}$. Let $\mathcal{O}_{D}$ be a maximal order stable under a neben involution $*$ of $D$. We denote by $\FF$ an algebraic closure of $\FF_{p}$ and by $W_{0}=W(\FF)$ the Witt vectors of $\FF$. Let $S$ be a $W_{0}$-scheme on which $p$ is locally nilpotent and $\bar{S}$ the closed subscheme defined by the ideal sheaf $p\mathcal{O}_{S}$. We fix a principally polarized $p$-divisible group $\XX$ over $\FF$ which is isoclinic of slope $1/2$ equipped with an action $\iota: \mathcal{O}_{D}\rightarrow \End(\XX)$. We consider the set valued functor $\mathcal{N}$ on those $S$-points above parametrizing the following data up to isomorphism.
\begin{enumerate}
\item $X_{0}$ is a $p$-divisible group over $S$ of dimension $4$.
\item $\iota_{0}: \mathcal{O}_{D}\rightarrow \End{X_{0}}$ is a $\mathcal{O}_{D}$-action on $X$.
\item $\lambda_{0}: X_{0}\rightarrow X_{0}^{\vee}$ is a principal polarization.
\item $\rho_{0}: X_{0}\times_{S} \bar{S}\rightarrow \XX\times_{\FF} \bar{S}$ is a quasi-isogeny. 
\end{enumerate}
The quadruple $(X_{0}, \iota_{0}, \lambda_{0}, \rho_{0})$ are required to satisfy certain additional conditions which we will refer the reader to Definition \ref{RZ-functor} for a precise definition. This functor is representable by a formal scheme locally formally of finite type. It decomposes into connected components $\mathcal{N}(i)$ for $i\in \ZZ$ according to the height of the quasi-isogeny $\rho_{0}$. Let $\mathcal{M}$ be the reduced scheme of the formal scheme $\mathcal{N}(0)$. The structure of this scheme $\mathcal{M}$ is studied using the Bruhat-Tits stratification in \cite{Wang-2}. In this article we are mainly concerned with the following variant of the above functor which we call the \emph{quaternionic unitary Rapoport-Zink space with Iwahori level structure}. We denote this functor by $\mathcal{N}_{\Iw}$ and its $S$-valued points is given by the isomorphism classes of the following data
\begin{enumerate}
\item $X_{i}$ for $i=-1, 0, 1$ is a $p$-divisible group over $S$ of dimension $4$ with $\mathcal{O}_{D}$-linear isogenies 
\begin{equation*}
X_{-1}\rightarrow X_{0}\rightarrow X_{1}.
\end{equation*}
\item $\iota_{i}: \mathcal{O}_{D}\rightarrow \End{X_{i}}$ is a $\mathcal{O}_{D}$-action on $X_{i}$ for $i=-1,0,1$.
\item $\lambda_{i}: X_{i}\rightarrow X^{\vee}_{-i}$ is an isomorphism.
\item $\rho_{i}: X_{i}\times_{S} \bar{S}\rightarrow \XX\times_{\FF} \bar{S}$ is a quasi-isogeny.  
\end{enumerate}
We refer the reader to Definition \ref{RZ-functor} for a precise and uniform way of defining these two moduli problems. Our main results concern the geometry of $\calM_{\Iw}$ defined in a completely similar way as $\calM$. We denote by $\pi_{\{1\}}: \calM_{\Iw}\rightarrow \calM$ the natural map. We first recall the main result of \cite{Wang-2} which describes the scheme $\calM$. 

\begin{thm}[\cite{Wang-2}]The scheme $\calM$ is pure of dimension $2$ and can be decomposed into 
$$\calM=\calM^{\circ}_{\{0\}}\sqcup\calM^{\circ}_{\{2\}}\sqcup\calM^{\circ}_{\{02\}}\sqcup\calM_{\{1\}}$$\
which is called the Bruhat-Tits stratification. 
\begin{enumerate}
\item Let $\calM_{\{0\}}$ be the closure of $\calM^{\circ}_{\{0\}}$. We denote by $\calM_{L_{0}}$ an irreducible component of $\calM_{\{0\}}$, then $\calM_{L_{0}}$ is isomorphic to a projective surface give by  
$$x^{p}_{3}x_{0}-x^{p}_{0}x_{3}+x^{p}_{2}x_{1}-x^{p}_{1}x_{2}=0.$$ 
Let $\calM_{\{2\}}$ be the closure of $\calM^{\circ}_{\{2\}}$. It is isomorphic to $\calM_{\{0\}}$ and an irreducible component of it is  denoted by $\calM_{L_{2}}$ which is isomorphic to the same surface as above.
\item An irreducible component of the closure $\calM_{\{02\}}$ of $\calM^{\circ}_{\{02\}}$ is denoted by $\calM_{L_{0}, L_{2}}$ and it is isomorphic to a projective line $\PP^{1}$.
\item The scheme $\calM_{\{1\}}$ consists of finitely many $\FF$-points which are superspecial.
\end{enumerate}
\end{thm}

The following theorem describes the structure of the space $\calM_{\Iw}$ and the proof relies heavily on the method used in \cite{Wang-2}.

\begin{thm}\label{thm2}
The scheme $\calM_{\Iw}$ is not equidimensional and it admits a decomposition of the form
$$\calM_{\Iw}=\mathcal{Y}_{\{0\}}\sqcup \mathcal{Y}_{\{2\}}\sqcup \mathcal{Y}_{\{02\}}\sqcup \mathcal{Y}_{\{1\}}$$
\begin{enumerate}
\item The scheme $\mathcal{Y}_{\{0\}}=\pi_{\{1\}}^{-1}(\calM^{\circ}_{\{0\}})$ is a $\PP^{1}$-bundle over $\calM^{\circ}_{\{0\}}$.
\item The scheme $\mathcal{Y}_{\{2\}}=\pi_{\{1\}}^{-1}(\calM^{\circ}_{\{2\}})$ is a $\PP^{1}$-bundle over $\calM^{\circ}_{\{2\}}$.
\item The scheme $\mathcal{Y}_{\{02\}}=\pi_{\{1\}}^{-1}(\calM^{\circ}_{\{2\}})$ is an $E$-bundle over $\calM^{\circ}_{\{2\}}$ where $E$ is the transverse intersection of two $\PP^{1}$ at a point.
\item The scheme $\mathcal{Y}_{\{1\}}=\pi_{\{1\}}^{-1}(\calM_{\{1\}})$ is a $\PP^{1}\times \PP^{1}$-bundle over $\calM_{\{1\}}$. 
\end{enumerate}
The scheme $\calM_{\Iw}$ has three types of irreducible components. We denote them by $\overline{\mathcal{Y}}_{L_{0}}$, $\overline{\mathcal{Y}}_{L_{2}}$ and ${\mathcal{Y}}_{L_{1}}$ respectively. 
\begin{enumerate}
\item $\overline{\mathcal{Y}}_{L_{0}}$, $\overline{\mathcal{Y}}_{L_{2}}$ are $\PP^{1}$-bundles over the surfaces $\mathcal{M}_{L_{0}}$ and $\mathcal{M}_{L_{2}}$ which are three dimensional
\item ${\mathcal{Y}}_{L_{1}}$ is isomorphic to  $\PP^{1}\times \PP^{1}$ which is two dimensional.
\end{enumerate}
\end{thm}
We provide an informal discussion of the method proving the above theorem. We will construct two $\PP^{1}$-bundles over $\calM$ using the data coming from $\calM_{\Iw}$. Over the stratum $\calM_{\{1\}}$, the two $\PP^{1}$-bundles will not interfere with each other. When we move to $\calM^{\circ}_{\{02\}}$, one $\PP^{1}$-bundle will determine the other one  except at one point and this gives the transverse intersection of two $\PP^{1}$ at a point. Over $\calM^{\circ}_{\{0\}}$ and $\calM^{\circ}_{\{2\}}$, one $\PP^{1}$-bundle will be completely determined by the other one. The following theorem describes the intersections of different types of irreducible components.
\begin{thm}\label{thm3}
The schemes $\overline{\mathcal{Y}}_{L_{0}}$, $\overline{\mathcal{Y}}_{L_{2}}$ and ${\mathcal{Y}}_{L_{1}}$ intersect with each other in the following way. 
\begin{enumerate}
\item The intersection of $\overline{\mathcal{Y}}_{L_{0}}$ and $\overline{\mathcal{Y}}_{L_{2}}$ is one dimensional and is equal to $\mathcal{M}_{L_{0},L_{2}}$.
\item The intersection of $\overline{\mathcal{Y}}_{L_{0}}$ and ${\mathcal{Y}}_{L_{1}}$ is one dimensional and is equal to a $\PP^{1}$-bundle over $\calM_{L_{1}}$.
\item The intersection of $\overline{\mathcal{Y}}_{L_{2}}$ and ${\mathcal{Y}}_{L_{1}}$ is one dimensional and is equal to a $\PP^{1}$-bundle over $\calM_{L_{1}}$.
\end{enumerate}
\end{thm}

As an application of the above results, we consider the following quaternionic unitary Rapoport-Zink space $\calN_{P}$. The set $\calN_{P}(S)$ classifies the isomorphism classes of the quadruple $(X_{i}, \lambda_{i}, \iota_{i}, \rho_{i})$ where
\begin{enumerate}
\item $X_{i}$ for $i=-1, 1$ is a $p$-divisible group over $S$ of dimension $4$ with $\mathcal{O}_{D}$-linear isogeny
 \begin{equation*}
 X_{-1}\rightarrow X_{1};
 \end{equation*}
\item $\iota_{i}: \mathcal{O}_{D}\rightarrow \End{X_{i}}$ is a $\mathcal{O}_{D}$-action on $X_{i}$ for $i=-1,1$;
\item $\lambda_{i}: X_{i}\rightarrow X^{\vee}_{-i}$ is an isomorphism for $i=-1,1$;
\item $\rho_{i}: X_{i}\times_{S} \bar{S}\rightarrow \XX\times_{\FF} \bar{S}$ is a quasi-isogeny for $i=-1,1$.  
\end{enumerate}
We will refer to this Rapoport-Zink space as the \emph{quaternionic unitary Rapoport-Zink space with Siegel parahoric level structure}.  Note that $\calN_{\Iw}$ can be viewed as a correspondence between $\calN$ and $\calN_{P}$. This means there is a natural diagram 
\begin{equation}\label{corres}
\begin{tikzcd}
& &\calN_{\Iw}\arrow{rd}{\pi_{\{02\}}} \arrow{ld}[swap]{\pi_{\{1\}}}\\
&\calN&  &\calN_{P}.\\ 
\end{tikzcd}
\end{equation}
Let $\calM_{P}$ be the underlying reduced scheme of $\calN_{P}(0)$ defined similarly as above. As an immediate application of Theorem \ref{thm2} and Theorem \ref{thm3} is the following description of the scheme $\calM_{P}$
\begin{thm}
The scheme $\calM_{P}$ is pure of dimension $2$. Moreover we have
\begin{equation}
\calM_{P}=\bigsqcup_{x\in \calM_{\{1\}}}(\PP^{1}\times\PP^{1})_{x}.
\end{equation} 
\end{thm}

We  remark that the set parametrizing the irreducible components of  $\calM_{P}$ corresponds to the set of superspecial points on $\calM$ which happens also to be the singular locus of $\calM$. Note this also happens for the usual Siegel modular threefold, see \cite[Theorem 4.7]{Yu06}.

\subsection{Results  on supersingular locus} Let $B$ be an indefinite quaternion algebra over $\QQ$. We assume that $B$ is equal to $D$ at the place $p$. Let $V=B\oplus B$ equipped with an alternating form $$(\cdot,\cdot): V\times V\rightarrow \QQ.$$ Let $G=\GU_{B}(V)$ be the quaternionic unitary group of $(G, (\cdot,\cdot))$. Note that $G_{\RR}$ is isomorphic to $\GSp(4)_{\RR}$ and let $h: \GG_{m,\CC}\rightarrow G_{\CC}$ be the Hodge cocharacter given by sending $z\in \GG_{m, \CC}$ to $\text{diag}(z,z,1,1)\in \GSp(4)_{\CC}$. Let $U^{p}\subset G(\mathbb{A}^{p}_{f})$ be an open compact subgroup which is sufficiently small. Let $U_{p}$ be the Iwahori subgroup or the Siegel parahoric subgroup of $G(\QQ_{p})$. In Section $5$, we will use these data to define  
integral models $\Sh_{\Iw, U^{p}}$ and $\Sh_{P, U^{p}}$ over $\ZZ_{p}$ of the quaternionic unitary Shimura varieties.  We denote by $Sh_{\Iw, U^{p}}$ the special fiber of $\Sh_{\Iw, U^{p}, W_{0}}$ and by $Sh^{ss}_{\Iw, U^{p}}$ the supersingular locus of $Sh_{\Iw, U^{p}}$.  Similarly we denote by $Sh_{P, U^{p}}$ the special fiber of $\Sh_{P, U^{p}, W_{0}}$ and by $Sh^{ss}_{P, U^{p}}$ the supersingular locus of $Sh_{P, U^{p}}$. The results of the previous subsection immediately apply to characterize $Sh^{ss}_{\Iw, U^{p}}$ and $Sh^{ss}_{P, U^{p}}$  via the Rapoport-Zink uniformization theorem \cite[Theorem 6.30]{RZ96}. 
\begin{thm}
The scheme $Sh^{ss}_{\Iw, U^{p}}$ has both $2$-dimensional components and $3$-dimensional components. The $2$-dimensional components are of the form $\PP^{1}\times \PP^{1}$ and a $3$-dimensional component is a $\PP^{1}$-bundle over the surface $x^{p}_{3}x_{0}-x^{p}_{0}x_{3}+x^{p}_{2}x_{1}-x^{p}_{1}x_{2}=0$. In particular since the dimension of $Sh_{\Iw, U^{p}}$ is $3$, there are irreducible components of it that are purely supersingular and thus the ordinary locus of $Sh_{\Iw, U^{p}}$ is not dense. 

The scheme $Sh^{ss}_{P, U^{p}}$ is purely two dimensional and an irreducible of it is isomorphic to $\PP^{1}\times\PP^{1}$. 

\end{thm}

\subsection{Notations and conventions}Let $p$ be a prime and let $\FF$ be an algebraically closed field containing $\FF_{p}$. Let $W_{0}=W(\FF)$ be the Witt vectors of $\FF$ and $K_{0}=W(\FF)_{\QQ}$ its fraction field. Denote by $\psi_{0}$ and $\psi_{1}$ the two embeddings of $\FF_{p^{2}}$ in $\FF$. Let $M_{1}\subset M_{2}$ be two $W(\FF)$-modules we wrire $M_{1}\subset^{d} M_{2}$ if the $\ZZ_{p}$-colength of the inclusion is $d$. If $R$ is ring and $L$ is an $R$-module and $R^{\prime}$ is an $R$-algebra, we define $L_{R^{\prime}}=L\otimes_{R} R^{\prime}$. Let $X$ be a (formal) scheme over $R$, we write $X_{R^{\prime}}$ its base change to $R^{\prime}$.

\subsection*{Acknowledgement} We would like to thank Michael Rapoport for suggesting the problem and make the preprint \cite{GHR19} available to us. We would like to thank Ulrich G\"{o}rtz, Liang Xiao and Eyal Goren for helpful discussions related to this work. This work is completed when the author is a postdoctoral fellow at McGill university and he would like to thank Henri Darmon and Pengfei Guan for their generous support. 

\section{Quaternionic unitary Rapoport-Zink spaces}
\subsection{Quaternionic unitary local Shimura datum}\label{local-shi-datum}
We fix an odd prime $p$. Let $\FF$ be an algebraically closed field containing $\FF_{p}$ and let $W_{0}=W(\FF)$ and $K_{0}=W(\FF)_{\QQ}$ be its fraction field. Let $D$ be the quaternion division algebra and $\mathcal{O}_{D}$ be a maximal order in $D$ stable under the a neben involution $*$. 
More precisely the involution $*$ can be given by 
\begin{equation*}
\begin{split}
&x^{*}=\sigma(x), x\in \QQ_{p^{2}}\\
&\Pi^{*}=\Pi.\\
\end{split}
\end{equation*}
We write $\Trd$ and $\Nrd$ the reduced trace and reduced norm from $D$ tp $\QQ_{p}$. We fix a presentation of $D=\QQ_{p^{2}}[\Pi]$ with $\Pi^{2}=p$ such that $\Pi a=\sigma(a)\Pi$. Let $V=D^{2}$ and fix an alternating form $$(\cdot,\cdot): V\times V\rightarrow \QQ_{p}$$ such that $(av_{1}, v_{2})=(v_{1}, a^{*}v_{2})$ for all $v_{1}, v_{2}\in V$ and $a\in D$. We define the algebraic group $G$ over $\QQ_{p}$ by
\begin{equation}
G(R)=\{g\in GL_{D}(V\otimes_{\QQ_{p}} R): (gx, gy)=c(g)(x, y)\text{ for some $c(g)\in R^{\times}$}\}.
\end{equation}
We refer to $G$ as the quaternionic unitary group and sometimes we write it as $\GU_{D}(V)$ or $\GU_{D}(2)$. By \cite[1.42]{RZ96}, $G$ is a non-trivial inner form of $\GSp(4)$. Let $b\in G(K_{0})$ be an element such that $c(b)=p$ and its image in the $\sigma$-conjugacy classes $B(G)$ of $G$ corresponds to an isocrystal $(N, F)=(V_{K_{0}}, b\sigma)$ which is isoclinic of slope $\frac{1}{2}$. Let $\XX$ be a $p$-divisible group whose associated isocrystal is $(N, F)$. The alternating form $(\cdot, \cdot)$ gives rise to a polarization $\lambda_{\XX}: \XX\rightarrow \XX^{\vee}$. 
For example we can choose $b$ to be the diagonal matrix with $\Pi$ on the diagonal. Since we have $\QQ_{p^{2}}\otimes K_{0}=K_{0}\oplus K_{0}$, there is a decomposition $N=N_{0}\oplus N_{1}$ with $FN_{0}=\Pi N_{0}=N_{1}$. Finally, using the isomorphism $G(K_{0})= \GSp(4)(K_{0})$, we define the \emph{ minscule cocharacter} $\mu: \mathbb{G}_{m}\rightarrow G(K_{0})$ by $z\rightarrow\text{diag}(z,z,1,1)$.

If $\Lambda\subset V_{K_{0}}$ is a lattice, then we define $\Lambda^{\perp}=\{x\in V: (x, y)\in W_{0} \text{ for all $y\in\Lambda$ } \}$ to be the dual lattice.  We consider \emph{periodic lattice chains} in $V$ \cite[Definition 3.]{RZ96}. More precisely, we consider the infinite chain $\mathcal{L}_{\emptyset}$ of $\mathcal{O}_{D}$ lattices in $V$ given by
\begin{equation}\label{chain-iw}
 \cdots\rightarrow \Lambda_{-2}=\Pi\Lambda_{0}\rightarrow \Lambda_{-1}\rightarrow \Lambda_{0}\rightarrow \Lambda_{1}\rightarrow \Lambda_{2}=\Pi^{-1}\Lambda_{0}\rightarrow \cdots.
\end{equation}
We also require that $\Lambda^{\perp}_{0}=\Lambda_{0}$ and $\Lambda^{\perp}_{1}=\Lambda_{-1}$ which means $\mathcal{L}_{\emptyset}$ is selfdual. Now we define the Iwahori subgroup $\Iw\subset G(\QQ_{p})$ by $$\Iw=\bigcap_{i}\Stab_{G(\QQ_{p})}(\Lambda_{i}).$$ 

The above discussion gives an \emph{integral Rapoport-Zink datum} $\mathcal{D}_{\Iw}=(\QQ_{p}, D, \mathcal{O}_{D}, V, (\cdot, \cdot), *,  b, \mu, \mathcal{L}_{\emptyset})$ \cite[Definition 3.18]{RZ96}, \cite[Remark 4.1]{RV14}. We will also consider the chain of lattices $\mathcal{L}_{\{1\}}$ given by
\begin{equation}\label{chain-para}
 \cdots\rightarrow \Lambda_{-2}=\Pi\Lambda_{0}\rightarrow  \Lambda_{0}\rightarrow \Lambda_{2}=\Pi^{-1}\Lambda_{0}\rightarrow \cdots.
\end{equation}
and its associated integral Rapoport-Zink datum $$\mathcal{D}_{1}=(\QQ_{p}, D, \mathcal{O}_{D}, V, (\cdot, \cdot), *,  b, \mu, \mathcal{L}_{\{1\}}).$$  
Notice that we can define a natural map \begin{equation}\label{piL}\pi_{\{1\}}: \calL_{\emptyset}\rightarrow \calL_{\{1\}}\end{equation} by taking the lattice chain $\calL_{\emptyset}$ and skipping those that are $\Pi$-power multiples of $\Lambda_{-1}$ and $\Lambda_{1}$.
Similarly, we can also consider the chain of lattices $\mathcal{L}_{\{02\}}$ given by  
\begin{equation}\label{chain-sieg}
 \cdots\Lambda_{-3}=\Pi\Lambda_{-1} \rightarrow \Lambda_{-1}\rightarrow  \Lambda_{1}\rightarrow \Lambda_{3}=\Pi^{-1}\Lambda_{1}\cdots.
\end{equation}
and its associated integral Rapoport-Zink datum $$\mathcal{D}_{\{02\}}=(\QQ_{p}, D, \mathcal{O}_{D}, V, (\cdot, \cdot), *,  b, \mu, \mathcal{L}_{\{02\}}).$$
We also have the natural map  \begin{equation}\label{piL}\pi_{\{02\}}: \calL_{\emptyset}\rightarrow \calL_{\{02\}}.\end{equation}  by taking the lattice chain $\calL_{\emptyset}$ and skipping those that are $\Pi$-power multiples of $\Lambda_{0}$.

\subsection{Rapoport-Zink space of parahoric level}\label{RZ-para}
The datum $\mathcal{D}_{*}$ with $*=\Iw, \{1\}, \{0,2\}$ gives rise to a \emph{quaternionic Rapoport-Zink space with parahoric level structure}. Let $\Nilp$ be the category of $W_{0}$-schemes $S$ on which $p$ is locally nilpotent. We denote by $\bar{S}$ the closed subscheme of $S$ defined by the ideal sheaf $p\mathcal{O}_{S}$. Let $X$ be a $p$-divisible group over $S$ and denote by $\DD(X)$ the \emph{covariant {\Dieu} crystal}  of $X$ cf. \cite{BBM82}. In the following $\calL=\calL_{\emptyset}, \calL_{\{0,2\}} \text{ or }\calL_{\{1\}}$.
\begin{definition}\label{RZ-functor}
We define the set valued functor $\calN_{\calL}: \Nilp\rightarrow \Sets$ such that for $S\in\Nilp$ the set $\calN_{\calL}(S)$ parametrizes the following data up to isomorphism.
\begin{enumerate}
\item For each lattice $\Lambda\in \mathcal{L}$, a $p$-divisible group $X_{\Lambda}$ over $S$ of dimension $4$ with an action $$\iota_{\Lambda}: \calO_{D}\rightarrow \End(X_{\Lambda}) .$$

\item For each $\Lambda\in \mathcal{L}$, a quasi-isogeny 
$$\rho_{\Lambda}: \XX\times_{\FF} \bar{S}\rightarrow X_{\Lambda}\times_{S} \bar{S}.$$

\item For each $\Lambda\in \mathcal{L}$, an isomorphism
$$\lambda_{\Lambda}: X_{\Lambda}\rightarrow X^{\vee}_{\Lambda^{\perp}}.$$
\end{enumerate}
We require that the datum $(X_{\Lambda}, \iota_{\Lambda}, \rho_{\Lambda}, \lambda_{\Lambda})$ satisfies the following conditions
\begin{enumerate}
\item Locally on $S$ the $\calO_{D}\otimes \calO_{S}$-module $M_{\Lambda}=\DD(X_{\Lambda})(S)$ is isomorphic to $\Lambda\otimes \calO_{S}.$
\item Let $\Lambda\subset \Lambda^{\prime}$ be neighbours in the chain $\calL$ and let $\tilde{\rho}_{\Lambda^{\prime}, \Lambda}$
be the quasi-isogeny that lifts $\rho_{\Lambda^{\prime}, \Lambda}=\rho_{\Lambda^{\prime}}\circ\rho^{-1}_{\Lambda}$. Then $\tilde{\rho}_{\Lambda^{\prime}, \Lambda}$ is an isogeny whose kernel $\Ker(\tilde{\rho}_{\Lambda^{\prime}, \Lambda})$ is locally isomorphic to $\Lambda^{\prime}/\Lambda$ on $S$.
\item For an element $a\in D^{\times}$ that normalizes $\calO_{D}$, let $X^{a}_{\Lambda}$ be the same $p$-divisible group $X_{\Lambda}$ but with an action $\iota^{a}_{X}: \calO_{D}\rightarrow \End(X_{\Lambda})$ given by $\iota^{a}_{X}(x)=\iota_{X}(a^{-1}xa)$. Then there exists an isomorphism $\theta_{a}: X^{a}_{\Lambda}\rightarrow X_{a\Lambda}$.
\item We require that $\iota_{\Lambda}$ satisfies the Kottwitz condition
$$\det(T-\iota_{\Lambda}(a); \Lie X_{\Lambda})= (T^{2}-\Trd(a))T+\Nrd(a))^{2}$$ for all $a\in \calO_{D}$ and all $\Lambda\in \calL$.

\item The polarization $\lambda_{\XX}: \XX_{\bar{S}}\rightarrow \XX^{\vee}_{\bar{S}}$ agrees with the composite $\rho^{\vee}_{\Lambda^{\perp}}\circ\lambda_{\Lambda, \bar{S}}\circ\rho_{\Lambda}$ up to a $\QQ^{\times}_{p}$-multiple. 
\end{enumerate}
\end{definition}

\begin{theorem}[{\cite[Theorem 3.25]{RZ96}}]
The functor $\calN_{\calL}$ is representable by a formal scheme over $W_{0}$ which is formally locally of finite type. 
\end{theorem}

We will refer to $\calN_{\Iw}=\calN_{\calL_{\emptyset}}$ as the \emph{quaternionic unitary Rapoport-Zink space with Iwahori level structure} and $\calN=\calN_{\calL_{\{1\}}}$ as the \emph{quaternionic unitary Rapoport-Zink space with paramodular level structure}. The space  $\calN_{P}=\calN_{\calL_{\{0,2\}}}$ will be referred to as the \emph{quaternionic unitary Rapoport-Zink space with Siegel parahoric level structure}.  The underlying reduced scheme of $\calN$ was studied in \cite{Wang-2} using the Bruhat-Tits stratification \cite{Vol10}, \cite{VW11}, \cite{GH15}. For $\calL=\calL_{\emptyset},\calL_{\{0,2\}} \text{ or }\calL_{\{1\}}$, we have the decomposition $$\calN_{\calL}=\bigsqcup_{i\in \ZZ} \calN^{(i)}_{\calL}$$ where $\calN^{(i)}_{\calL}$ is the open and closed formal subscheme  of $\calN_{\calL}$ classifying those $x=(X_{\Lambda}, \iota_{\Lambda},  \rho_{\Lambda}, \lambda_{\Lambda})$ such that $\rho^{\vee}_{\Lambda}\circ\lambda_{\Lambda, \bar{S}}\circ\rho_{\Lambda}=c(x)\lambda_{\XX}$ for some $c(x)\in p^{i}\ZZ^{\times}_{p}$ and a fixed $\Lambda\in \calL$. 

Let $J_{b}$ be the group functor assigning any $\QQ_{p}$-algebra $R$ the group $$J_{b}(R)=\{g\in G(R\otimes_{\QQ_{p}}K_{0}): g(b\sigma)=(b\sigma)g\}.$$
It is representable by an algebraic group over $\QQ_{p}$ and is in our case isomorphic to $\GSp(4)$. We review the proof of this fact as it will be useful when we discuss the set of geometric points of $\calN_{\calL}$. Recall that we have the isocrystal $(N, F)$ associated to $\XX$ and the decomposition $N=N_{0}\oplus N_{1}$. Let $\tau=\Pi^{-1}F$ and consider its action on $N_{0}$. Then we have $N_{0}= C\otimes K_{0}$ with $C=N^{\tau=1}_{0}$. We define a new alternating form $(\cdot, \cdot)_{0}$ on $N_{0}$ by $(x, y)_{0}=(x, \Pi y)$ for all $x,y\in N_{0}$. We have
\begin{equation}
\begin{split}
(\tau x, \tau y)_{0}= &(\Pi^{-1}F x, \Pi^{-1}F y)\\
&=\sigma(x, \Pi y)\\
&=\sigma(x, y)_{0}.\\
\end{split}
\end{equation}
for all $x,y\in N_{0}$. We can restrict the form $(\cdot,\cdot)_{0}$ on $C$ and consider $(C, (\cdot, \cdot)_{0})$ as a symplectic space over $\QQ_{p}$. 

\begin{lemma}
There is an isomorphism $J_{b}(\QQ_{p})\cong \GSp(C)$.
\end{lemma}
\begin{proof}
By definition we have $J_{b}(\QQ_{p})=\{g\in \GU_{D}(N); gF=Fg\}$. Since $g$ commutes with the action of $D$, it respects the decomposition of $N=N_{0}\oplus N_{1}$.  Since it also commutes with the Frobenius $F$, the action is uniquely determined on its restriction to $N_{0}$. The result then follows from the fact that $g$ also commutes with $\tau$. 
\end{proof}

The group $J_{b}(\QQ_{p})$ acts on the space $\calN_{\calL}$ from left. Let $x=(X_{\Lambda}, \iota_{\Lambda}, \rho_{\Lambda}, \lambda_{\Lambda})$, then we define $g.x=(X_{\Lambda}, \iota_{\Lambda}, \rho_{\Lambda}\circ g^{-1}, \lambda_{\Lambda})$. For each $g\in J_{b}(\QQ_{p})$, this action restricts to an isomorphism $\calN^{(i)}_{\calL}\cong\calN^{(i+\ord_{p}(c(g)))}_{\calL}$ where $c(g)$ is the similitude factor of $g\in \GSp(C)$. Since the similitude map $c: \GSp(C)\rightarrow \QQ^{\times}_{p}$ is surjective and $c(p)=p^{2}$, we arrive at the following lemma immediately. 

\begin{lemma}\hfill
\begin{enumerate}
\item There is an isomorphism $\calN^{(i)}_{\calL}\cong\calN^{(0)}_{\calL}$ for any $i\in \ZZ$.
\item  There is an isomorphism $p^{\ZZ}\backslash \calN_{\calL}=\calN^{(0)}_{\calL}\sqcup \calN^{(1)}_{\calL}$.
\end{enumerate}
\end{lemma}

\subsection{Lattice description of geometric points} Let $\calM_{\calL}$ be the underlying reduced scheme of the the formal scheme $\calN^{(0)}_{\calL}$. The goal of this subsection is to describe the set $\calM_{\calL}(\FF)$ of geometric points of $\calM_{\calL}$ in terms of lattices in $N_{0}$. 
We first discuss the case of $\calL=\calL_{\{1\}}$ and we write $\calM_{\calL_{\{1\}}}$ as $\calM$. We have repeated some proofs in \cite{Wang-2} for the reader's convenience. 
\begin{lemma}\label{RZ-set-pre}
The map sending a point $x=(X_{\Lambda}, \iota_{\Lambda}, \rho_{\Lambda}, \lambda_{\Lambda})\in\calM(\FF)$ to the {\Dieu} module $M=M_{\Lambda_{0}}$ of $X_{\Lambda_{0}}$ gives a bijection between the set $\calM$ and the set of $W_{0}$-lattices
$$\{M\subset N: M^{\perp}=M, pM\subset VM\subset^{4} M, \Pi M\subset^{4} M\}.$$
\end{lemma}

\begin{proof}
Since any $\Lambda\in \calL_{\{1\}}$ is of the form $\Pi^{n} \Lambda_{0}$ for some $n\in \ZZ$, $X_{\Lambda}$ is determined by $X_{\Lambda_{0}}$ by condition $(3)$ in Definition \ref{RZ-functor}. Let $M=M_{\Lambda_{0}}$, then $M=M^{\perp}$ follows from condition $(1)$ in Defintion \ref{RZ-functor} and the fact that $\Lambda_{0}=\Lambda^{\perp}_{0}$.  The inclusions $pM\subset VM\subset^{4} M $ follows from the Kottwitz condition which is condition $(4)$ in Definition \ref{RZ-functor}. The inclusion $ \Pi M\subset^{4} M$ follows from condition $(2)$ in Definition \ref{RZ-functor}. Conversely given $M$ in the above set, it clearly determines a tuple $(X_{\Lambda_{0}}, \iota_{\Lambda_{0}}, \rho_{\Lambda_{0}}, \lambda_{\Lambda_{0}})$ which in turn determines a unique point $x=(X_{\Lambda}, \iota_{\Lambda}, \rho_{\Lambda}, \lambda_{\Lambda})\in \calM(\FF)$.  
\end{proof}

Recall that we have $W_{0}\otimes \ZZ_{p^{2}}= W_{0}\oplus W_{0}$ and $\calO_{D}= \ZZ_{p^{2}}[\Pi]$. This induces a decomposition $M=M_{0}\oplus M_{1}$ for any $\calO_{D}$-module $M\subset N$. Moreover if we assume that $M$ is a {\Dieu} module, then we have 
\begin{equation}\label{M0M1}
\begin{split}
& \Pi M_{0}\subset M_{1}, \Pi M_{1}\subset M_{1};\\
& F M_{0}\subset M_{1}, F M_{1}\subset M_{1}.\\
\end{split}
\end{equation}

\begin{corollary}
The map sending $M$ to $M_{0}$ defines a bijection between the set $$\{M\subset N: M^{\perp}=M, pM\subset VM\subset^{4} M, \Pi M\subset^{4} M\}$$ in Lemma \ref{RZ-set-pre} and the set
\begin{equation}\label{RZ-set-para}
\{M_{0}\subset N_{0}:pM^{\vee}_{0}\subset M_{0}\subset^{2} M^{\vee}_{0}, p\tau(M^{\vee}_{0})\subset M_{0}\subset^{2} \tau(M^{\vee}_{0}) \}.
\end{equation}

Here $M^{\vee}_{0}$ is the dual lattice of $M_{0}$ with respect to the form $(\cdot, \cdot)_{0}$ on $N_{0}$. 
\end{corollary}

\begin{proof}
It follows from $\Pi M\subset^{4} M$ that $\Pi M_{1}\subset^{2} M_{0} \subset^{2}  \Pi^{-1}M_{1}$. Since $M^{\perp}=M$, $M^{\perp}_{0}=M_{1}$ and $M^{\perp}_{1}=M_{0}$.  Note that 
\begin{equation}
\begin{split}
\Pi M^{\perp}_{0}&= \{x\in N_{0}: (\Pi^{-1}x, y)\in W_{0} \text{ for all $y\in M_{0}$}\}\\
&=\{x\in N_{0}: (p^{-1}x, \Pi y)\in W_{0} \text{ for all $y\in M_{0}$}\}\\
&=\{x\in N_{0}: (p^{-1}x,  y)_{0}\in W_{0} \text{ for all $y\in M_{0}$}\}\\
&=pM^{\vee}_{0}.
\end{split}
\end{equation}
The inclusion $pM^{\vee}_{0}\subset M_{0}\subset^{2} M^{\vee}_{0}$ follows immediately.  The inclusion $p\tau(M^{\vee}_{0})\subset M_{0}\subset^{2} \tau(M^{\vee}_{0})$ follows from $pM_{0}\subset^{2} VM_{1}\subset^{2} M_{0}$ and $pM_{1}\subset^{2} VM_{0}\subset^{2} M_{1}$ by a similar computation as above.  Conversely if we are given $$M_{0}\in \{M_{0}\subset N_{0}:pM^{\vee}_{0}\subset M_{0}\subset^{2} M^{\vee}_{0}, p\tau(M^{\vee}_{0})\subset M_{0}\subset^{2} \tau(M^{\vee}_{0}) \},$$ then we can define $M= M_{0}\oplus \Pi M^{\vee}_{0}$ and check that it indeed gives an element in $$\{M\subset N: M^{\perp}=M, pM\subset VM\subset^{4} M, \Pi M\subset^{4} M\}.$$
These two maps are clearly inverse to each other and we get the claimed bijections. 
\end{proof}

We also have the following lemma.

\begin{lemma}\label{spin}
Let $M_{0}\in \calM(\FF)$ described in \eqref{RZ-set-para}, then $\dim_{\FF} (M_{0}+\tau(M_{0}))/M_{0}\leq 1$.
\end{lemma}
\begin{proof}
Let $M=M_{0}\oplus M_{1}$ be the associated {\Dieu} module. Since we have the composite of $\Pi: M_{0}/VM_{1}\rightarrow M_{1}/VM_{0}$ and $\Pi: M_{1}/VM_{0}\rightarrow M_{0}/VM_{1}$ being multiplication by $p$, at least one of them is not invertible. Note also that we have $(\Pi x, y )=(x, \Pi y)$, the other is also not invertible. On the level of $\FF$-points, this translates to 
\begin{equation*}
\dim_{\FF} VM_{0}+\Pi M_{0}/VM_{0} \leq1 \text{ and }
\dim_{\FF} VM_{1}+\Pi M_{1}/VM_{1} \leq1.
\end{equation*}
\end{proof}

 Next we move on to the Iwahori case $\calL=\calL_{\emptyset}$ and we write $\calM_{\Iw}= \calM_{\calL_{\emptyset}}$. 

\begin{lemma}\label{RZ-set-pre}
The map sending a point $x=(X_{\Lambda}, \iota_{\Lambda}, \rho_{\Lambda}, \lambda_{\Lambda})\in\calM_{\Iw}(\FF)$ to the {\Dieu} modules $$S=M_{\Lambda_{-1}}\subset M=M_{\Lambda_{0}}\subset T=M_{\Lambda_{1}}$$ of $X_{\Lambda_{-1}}\rightarrow X_{\Lambda_{0}}\rightarrow X_{\Lambda_{1}}$ gives a bijection between the set $\calM_{\Iw}(\FF)$ and the set of chain of $W_{0}$-lattices
\begin{equation}\label{Iw-set-pre}
\begin{split}
\{S\subset M\subset T\subset N: & M^{\perp}=M, S=T^{\perp}, \Pi M\subset^{2} S\subset^{2} M \subset^{2} T\subset^{2} \Pi^{-1} M,\\
 & pM\subset VM\subset^{4} M,    pS\subset VS\subset^{4} S,  pT\subset VT\subset^{4} T \}.\\
\end{split}
\end{equation}
\end{lemma}

\begin{proof}
Since any lattice $\Lambda\in \calL_{\emptyset}$ is of the form $\Lambda=\Pi^{n}\Lambda_{0}, \Pi^{n}\Lambda_{-1}, \Pi^{n}\Lambda_{1}$ for some $n\in \ZZ$, a point $x=(X_{\Lambda}, \iota_{\Lambda}, \rho_{\Lambda}, \lambda_{\Lambda})$ is determined by $(X_{\Lambda_{i}}, \iota_{\Lambda_{i}}, \rho_{\Lambda_{i}}, \lambda_{\Lambda_{i}})_{i=-1, 0, 1}$ by condition $(3)$ in Definition \ref{RZ-functor}. The equalities $M^{\perp}=M, S=T^{\perp}$ follows from $\Lambda_{0}=\Lambda^{\perp}_{0}, \Lambda_{-1}=\Lambda^{\perp}_{1}$ and condition $(1)$ in Lemma \ref{RZ-functor}.  The inclusions $\Pi M\subset^{2} S\subset^{2} M \subset^{2} T\subset^{2} \Pi^{-1} M$ follows from condition $(2)$ in Lemma \ref{RZ-functor}. The other three inclusions  $pM\subset VM\subset^{4} M,    pS\subset VS\subset^{4} S,  pT\subset VT\subset^{4} T$ follow from the Kottwitz conditions on $\iota_{\Lambda_{i}}$ for $ i=-1,0,1$. Conversely given an element $(S\subset M\subset T)$ in the set \eqref{Iw-set-pre}, it determines a tuple $(X_{\Lambda_{i}}, \iota_{\Lambda_{i}}, \rho_{\Lambda_{i}}, \lambda_{\Lambda_{i}})_{i=-1, 0, 1}$ satisfying all the conditions in Lemma \ref{RZ-functor} and hence determines a unique point in $\calM_{\Iw}(\FF)$ by the above discussion.
\end{proof}

\begin{corollary}
The map sending $S\subset M\subset T$ in \eqref{Iw-set-pre} to their $0$-th components $S_{0}\subset M_{0}\subset T_{0}$ defines a bijection between the set in \eqref{Iw-set-pre} and the set 
\begin{equation}
\{S_{0}\subset M_{0}\subset T_{0}\subset N_{0}: \begin{array}[c]{ccccccccc}
\tau(M_{0})&{\subset^{1}}&\tau(T^{\vee}_{0})&{\subset^{1}}&\tau(M^{\vee}_{0})&{\subset^{1}}&\tau(S^{\vee}_{0})&{\subset^{1}}&\frac{1}{p}\tau(M_{0})\\
\rotatebox{90}{$\subset$}&&\rotatebox{90}{$\subset$}&&\rotatebox{90}{$\subset$}&&\rotatebox{90}{$\subset$} &&\rotatebox{90}{$\subset$}\\
pM^{\vee}_{0}&{\subset^{1}}&S_{0}&{\subset^{1}} &M_{0} &{\subset^{1}} &T_{0} &{\subset^{1}}& M^{\vee}_{0}\\
\rotatebox{90}{$\subset$}&&\rotatebox{90}{$\subset$}&&\rotatebox{90}{$\subset$}&&\rotatebox{90}{$\subset$} &&\rotatebox{90}{$\subset$}\\
p\tau(M_{0})&{\subset^{1}}&p\tau(T^{\vee}_{0})&{\subset^{1}}&p\tau(M^{\vee}_{0})&{\subset^{1}}&p\tau(S^{\vee}_{0})&{\subset^{1}}&\tau(M_{0})
\end{array} \}.
\end{equation}
Here all the vertical inclusions have index $2$. 
\end{corollary}
\begin{proof}
Here the proof is essentially the same as the one given in Lemma \ref{RZ-set-para}. Notice that all the horizontal rows come from the requirement that $\Pi M\subset^{2} S\subset^{2} M \subset^{2} T\subset^{2} \Pi^{-1} M$ in \eqref{Iw-set-pre} and all the vertical inclusions come from the Kottwitz conditions imposed on $S\subset M\subset T$. Here one also uses the fact that $T^{\perp}=S$. 
\end{proof}

\subsection{Bruhat-Tits stratification of $\calM$} We review the theory of Bruhat-Tits stratification for $\calM$ in \cite{Wang-2}. Consider the derived group $J^{\der}_{b}=\Sp(C)$ of $J_{b}$ and denote by $\mathcal{B}(J^{\der}_{b})$ the \emph{Bruhat-Tits building} of $J^{\der}_{b}$. We fix a base alcove in $\mathcal{B}(J^{\der}_{b})$ and identify the nodes on the affine Dynkin diagram with the vertices of the chosen base alcove. The affine Dynkin diagram in our case is of type $\tilde{C}_{2}$ and is given by
\begin{displaymath}
\xymatrix{\underset{0}\circ \ar@2{->}[r] &\underset{1}\circ &\underset{2}\circ\ar@2{->}[l]}
\end{displaymath}
where the nodes $0$ and $2$ are special and $1$ is not special in the sense that the parahoric subgroup corresponding to $0$ and $2$ are special parahoric subgroups. There is an interpretation of the vertices in terms of lattices in $C$ and this gives the notion of \emph{vertex lattices}. Let $L$ be a lattice in $C$ then we three types of vertex lattices
\begin{equation}
\begin{split}
&\text{vertex lattice of type $0$: } pL^{\vee}\subset^{4} L\subset^{0} L^{\vee};\\
&\text{vertex lattice of type $2$: }  pL^{\vee}\subset^{0} L\subset^{4} L^{\vee};\\
&\text{vertex lattice of type $1$: }  pL^{\vee}\subset^{2} L\subset^{2} L^{\vee}.\\
\end{split}
\end{equation}
For each vertex lattice $L_{i}$ of type $i$ with $i=0, 1, 2$ we can define a $p$-divisible group $\mathbb{X}_{L^{+}_{i}}$ and $\mathbb{X}_{L^{-}_{i}}$and a projective scheme $\calM_{L_{i}}$ such that 
\begin{enumerate}
\item $\calM_{L_{0}}$ classifies the chain of isogenies 
$$X_{\Pi\Lambda_{0}}\rightarrow \mathbb{X}_{L^{-}_{0}}\rightarrow X_{\Lambda_{0}}\rightarrow \mathbb{X}_{L^{+}_{0}}\rightarrow X_{\Pi^{-1}\Lambda_{0}}$$
\item $\calM_{L_{2}}$ classifies the chain of isogenies
$$X_{\Pi\Lambda_{0}}\rightarrow \mathbb{X}_{L^{+}_{2}}\rightarrow X_{\Lambda_{0}}\rightarrow \mathbb{X}_{L^{-}_{2}}\rightarrow X_{\Pi^{-1}\Lambda_{0}}.$$
\item $\calM_{L_{1}}$ classifies those $X_{\Lambda_{0}}$ such that
$$ \mathbb{X}_{L^{+}_{1}}=X_{\Lambda_{0}}=\mathbb{X}_{L^{-}_{1}}.$$
\end{enumerate}
We will refer the reader to \cite[Section 4]{Wang-2} for the definition of the $p$-divisible group $\XX_{L}$ for a vertex lattice $L$. 
The $\FF$-points of theses schemes are given by
\begin{equation}
\begin{split}
&\calM_{L_{0}}(\FF)=\{M_{0}\in\calM(\FF): pL^{\vee}_{0, W_{0}}\subset^{1}pM^{\vee}_{0}\subset^{2}M_{0}\subset^{1}L_{0, W_{0}}\};\\
&\calM_{L_{2}}(\FF)=\{M_{0}\in\calM(\FF): L_{2, W_{0}}\subset^{1}M_{0}\subset^{2}M^{\vee}_{0}\subset^{1}L^{\vee}_{2, W_{0}}\};\\
&\calM_{L_{1}}(\FF)=\{M_{0}\in\calM(\FF): pL^{\vee}_{1, W_{0}}\subset^{0}pM^{\vee}_{0}\subset^{2}M_{0}\subset^{0}L_{1, W_{0}}\}.\\
\end{split}
\end{equation}

Those projective schemes of the form $\calM_{L_{0}}$ and $\calM_{L_{2}}$ are the irreducible components of the scheme $\calM$. They are smooth of dimension $2$. The scheme $\calM_{L_{1}}$ is of dimension $0$ and consists of a  reduced point which can be shown to be \emph{superspecial}. We will refer to these projective schemes as \emph{lattice strata}. 
Next we take the union of the lattice strata and we obtain the \emph{closed Bruhat-Tits strata}.

\begin{equation}
\begin{split}
&\calM_{\{0\}}=\bigcup_{L_{0}}\calM_{L_{0}} \text{ where $L_{0}$ goes through all the vertex lattices of type $0$};\\
&\calM_{\{2\}}=\bigcup_{L_{2}}\calM_{L_{2}} \text{ where $L_{2}$ goes through all the vertex lattices of type $2$};\\
&\calM_{\{1\}}=\bigcup_{L_{1}}\calM_{L_{1}} \text{ where $L_{1}$ goes through all the vertex lattices of type $1$}.\\
\end{split}
\end{equation}
We also record another useful characterization of these strata.

\begin{lemma}\label{BT-description}
\begin{equation}
\begin{split}
&\calM_{\{0\}}(\FF)=\{M_{0}\in \calM: \tau(M_{0}+\tau(M_{0}))=M_{0}+\tau(M_{0})\};\\
&\calM_{\{2\}}(\FF)=\{M_{0}\in \calM: \tau(M_{0}\cap\tau(M_{0}))=M_{0}\cap\tau(M_{0})\};\\
&\calM_{\{1\}}(\FF)=\{M_{0}\in\calM: \tau(M_{0})=M_{0}\}.\\
\end{split}
\end{equation}
\end{lemma}
\begin{proof}
This follows from the proof of \cite[Proposition 3.11]{Wang-2}.
\end{proof}

The scheme $\calM_{L_{0}}$ admits the \emph{Ekedhal-Oort stratification} given by $$\calM_{L_{0}}=\calM^{\circ}_{L_{0}}\sqcup \calM^{\circ}_{L_{0},\{2\}}\sqcup \calM_{L_{0}, \{1\}}$$ where 

\begin{equation}\label{lattice-stra}
\begin{split}
&\calM^{\circ}_{L_{0}}=\calM_{L_{0}}-(\calM_{\{2\}}\cup\calM_{\{1\}});\\
&\calM^{\circ}_{L_{0}, \{2\}}=(\calM_{L_{0}}\cap\calM_{\{2\}})-\calM_{\{1\}};\\
&\calM_{L_{0}, \{1\}}=\calM_{L_{0}}\cap\calM_{\{1\}}.\\
\end{split}
\end{equation}
Note an irreducible component of $\calM^{\circ}_{L_{0},\{2\}}$ is given by 
\begin{equation}
\begin{split}
\calM^{\circ}_{L_{0},L_{2}}=&\{M_{0}\in\calM_{L_{0}}: L_{2,W_{0}}\subset^{1} M_{0}\subset^{1} L_{0, W_{0}}\\
&\text{ and $M_{0}\neq L_{1, W_{0}}$ for any vertex lattice $L_{1}$ of type $1$}\}.\\
\end{split}
\end{equation}
Similarly the scheme $\calM_{L_{2}}$ admits the \emph{Ekedhal-Oort stratification} given by $$\calM_{L_{2}}=\calM^{\circ}_{L_{2}}\sqcup \calM^{\circ}_{L_{2},\{0\}}\sqcup \calM_{L_{2}, \{1\}.}$$
Here the strata are defined similarly as in the case of  $\calM_{L_{0}}$. 
We define next the open Bruhat-Tits strata by
\begin{equation}\label{BT-strata}
\begin{split}
&\calM^{\circ}_{\{0\}}=\calM_{\{0\}}-\calM_{\{2\}};\\
&\calM^{\circ}_{\{2\}}=\calM_{\{2\}}-\calM_{\{0\}};\\
&\calM^{\circ}_{\{02\}}=(\calM_{\{0\}}\cap\calM_{\{2\}})-\calM_{\{1\}}.\\
\end{split}
\end{equation}

The following theorem summarizes the above discussion about the geometry of the scheme of $\calM$. 

\begin{theorem}
The scheme $\calM$ can be decomposed as
$$\calM=\calM^{\circ}_{\{0\}}\sqcup\calM^{\circ}_{\{2\}}\sqcup\calM^{\circ}_{\{02\}}\sqcup\calM_{\{1\}}$$\
which is called the Bruhat-Tits stratification. 
\begin{enumerate}
\item The closure of the scheme $\calM^{\circ}_{\{0\}}$ is $\calM_{\{0\}}$ and the complement $$\calM_{\{0\}}-\calM^{\circ}_{\{0\}}=\calM^{\circ}_{\{02\}}\sqcup\calM_{\{1\}}.$$ The irreducible components of $\calM_{\{0\}}$ are of the form $$x^{p}_{3}x_{0}-x^{p}_{0}x_{3}+x^{p}_{2}x_{1}-x^{p}_{1}x_{2}=0$$ for a projective coordinate $[x_{0}:x_{1}:x_{2}:x_{3}]$ of $\PP^{3}$. 
\item  The scheme $\calM^{\circ}_{\{2\}}$ is isomorphic to $\calM^{\circ}_{\{0\}}$ and therefore the irreducible components of $\calM_{\{2\}}$ are of the form $$x^{p}_{3}x_{0}-x^{p}_{0}x_{3}+x^{p}_{2}x_{1}-x^{p}_{1}x_{2}=0$$ for a projective coordinate $[x_{0}:x_{1}:x_{2}:x_{3}]$ of $\PP^{3}$. 
\item  The closure of the scheme $\calM^{\circ}_{\{02\}}$ is $\calM_{\{02\}}$ whose irreducible components are isomorphic to $\PP^{1}$. The complement of $\calM_{\{02\}}-\calM^{\circ}_{\{02\}}$ is $\calM_{\{1\}}$ which consists of all the superspecial points on $\calM$.
\end{enumerate}
\end{theorem}

\section{The structure of the scheme $\calM_{\Iw}$} 
We now move on to describe the space $\calM_{\Iw}$. We denote by $$\pi_{\{1\}}: \calM_{\Iw}\rightarrow \calM$$ the natural map.  If $x=(X_{\Lambda}, \iota_{\Lambda}, \rho_{\Lambda}, \lambda_{\Lambda})\in \calM_{\Iw}$ with $\Lambda\in \calL_{\emptyset}$, then $\pi_{\{1\}}(x)= (X_{\Lambda^{\prime}}, \iota_{\Lambda^{\prime}}, \rho_{\Lambda^{\prime}}, \lambda_{\Lambda}^{\prime})$ with $\Lambda^{\prime}=\pi_{\{1\}}(\Lambda)$ as defined in \eqref{piL}. Our strategy to describe $\calM_{\Iw}$ is to describe the fiber of the Bruhat-Tits strata of $\calM$ under $\pi_{\{1\}}$. Therefore we set 
\begin{equation}
\begin{split}
&\mathcal{Y}_{\{1\}}= \pi_{\{1\}}^{-1}(\calM_{\{1\}});\\
&\mathcal{Y}_{\{0\}}= \pi_{\{1\}}^{-1}(\calM^{\circ}_{\{0\}});\\
&\mathcal{Y}_{\{2\}}= \pi_{\{1\}}^{-1}(\calM^{\circ}_{\{2\}});\\
&\mathcal{Y}_{\{02\}}= \pi_{\{1\}}^{-1}(\calM^{\circ}_{\{02\}}).
\end{split}
\end{equation}

\subsection{The description of $\mathcal{Y}_{\{1\}}$} Let $x\in \calM_{\{1\}}(\FF)$ be a $\FF$-point. We would like to first understand the set  $\pi^{-1}_{\{1\}}(x)$. 

\begin{lemma}\label{fiber1}
There is a bijection between $\pi_{\{1\}}^{-1}(x)$ and $\PP^{1}(\FF)\times\PP^{1}(\FF)$. 
\end{lemma} 
\begin{proof}
By Lemma \ref{RZ-set-para}, $x$ corresponds to $M_{0}$ that fits in 
\begin{equation}
\begin{split}
& pM^{\vee}_{0}\subset^{2} M_{0} \subset^{2} M^{\vee}_{0}\\
& p\tau(M^{\vee}_{0})\subset^{2} M_{0} \subset^{2} \tau(M^{\vee}_{0}).\\
\end{split}
\end{equation}

Then we need to find all possible $S_{0}$ and $T_{0}$ that fit in the following web
\begin{equation}\label{web}
\begin{array}[c]{ccccccccc}
\tau(M_{0})&{\subset^{1}}&\tau(T^{\vee}_{0})&{\subset^{1}}&\tau(M^{\vee}_{0})&{\subset^{1}}&\tau(S^{\vee}_{0})&{\subset^{1}}&\frac{1}{p}\tau(M_{0})\\
\rotatebox{90}{$\subset$}&&\rotatebox{90}{$\subset$}&&\rotatebox{90}{$\subset$}&&\rotatebox{90}{$\subset$} &&\rotatebox{90}{$\subset$}\\
pM^{\vee}_{0}&{\subset^{1}}&S_{0}&{\subset^{1}} &M_{0} &{\subset^{1}} &T_{0} &{\subset^{1}}& M^{\vee}_{0}\\
\rotatebox{90}{$\subset$}&&\rotatebox{90}{$\subset$}&&\rotatebox{90}{$\subset$}&&\rotatebox{90}{$\subset$} &&\rotatebox{90}{$\subset$}\\
p\tau(M_{0})&{\subset^{1}}&p\tau(T^{\vee}_{0})&{\subset^{1}}&p\tau(M^{\vee}_{0})&{\subset^{1}}&p\tau(S^{\vee}_{0})&{\subset^{1}}&\tau(M_{0})
\end{array} 
\end{equation}

Given $y\in \pi^{-1}(x)$, it determines $S_{0}$ and $T_{0}$ as above. Then sending them to 
\begin{equation*}
\begin{split}
&\bar{S}_{0}=S_{0}/pM^{\vee}_{0}\subset M_{0}/pM^{\vee}_{0}\\ 
&\bar{T}_{0}=T_{0}/M_{0}\subset M^{\vee}_{0}/M_{0}.\\
\end{split}
\end{equation*}
gives a map from $\pi^{-1}(x)$ to $\PP(M_{0}/pM^{\vee}_{0})(\FF)\times\PP(M^{\vee}_{0}/M_{0})(\FF)$. 

Conversely given $(\bar{S}_{0}, \bar{T}_{0})\in \PP(M_{0}/pM^{\vee}_{0})\times\PP(M^{\vee}_{0}/M_{0})(\FF)$, let 
\begin{equation*}
\begin{split}
&pM^{\vee}_{0}\subset S_{0}\subset M_{0}\\
&M_{0}\subset T_{0}\subset M^{\vee}_{0}\\
\end{split}
\end{equation*}
be the preimage of 
\begin{equation*}
\begin{split}
&\bar{S}_{0}\subset M_{0}/pM^{\vee}_{0} \\
& \bar{T}_{0}\subset M^{\vee}_{0}/M_{0}\\
\end{split}
\end{equation*}
under the natural reduction map. Since $M_{0}\subset T_{0}\subset M^{\vee}_{0}$, $p\tau(M_{0})\subset p\tau(T^{\vee}_{0})\subset p\tau(M^{\vee}_{0})$.
Since $M_{0}\in \calM_{\{1\}}$, we have $\tau(M_{0})=M_{0}$. It follows that $p\tau(T^{\vee}_{0})\subset p M^{\vee}_{0}\subset S_{0}$.  Similarly, since $pM^{\vee}_{0}\subset S_{0}\subset M_{0}$, $S_{0}\subset M_{0}=\tau(M_{0})\subset \tau(T^{\vee}_{0})$. This verifies the inclusion $p\tau(T^{\vee}_{0})\subset S_{0}\subset  \tau(T^{\vee}_{0})$. One verifies the inclusion $p\tau(S^{\vee}_{0})\subset T_{0}\subset S^{\vee}_{0}$ completely in a similar way. 
This finishes the proof as $S_{0}$ and $T_{0}$ fit in the web above and the two maps we have defined are mutually inverse to each other. 
\end{proof}

\begin{remark}\label{windremark}
Using the theory of windows for displays of $p$-divisible groups \cite{Zink-pro99},  one can replace $\FF$ by any field extension $k$ in the above lemma. 
\end{remark}

We now describe the global structure of $\mathcal{Y}_{\{1\}}$.  Let $L_{1}$ be a vertex lattice of type $1$, we denote by $\mathcal{Y}_{L_{1}}=\pi^{-1}_{\{1\}}(\mathcal{M}_{L_{1}})$. 

\begin{theorem}\label{YL1}
Let $\mathcal{Y}^{\prime}_{\{1\}}$ be the $\PP^{1}\times \PP^{1}$-bundle over $\calM_{\{1\}}$. There is an isomorphism between $\mathcal{Y}_{\{1\}}$ and $\mathcal{Y}^{\prime}_{\{1\}}$.
\end{theorem}

\begin{proof}
We first define a morphism $\alpha_{\{1\}}: \mathcal{Y}_{\{1\}}\rightarrow \mathcal{Y}^{\prime}_{\{1\}}$.  Let $R$ be a $\FF$-algebra. Let $y=(X_{\Lambda}, \iota_{\Lambda}, \rho_{\Lambda}, \lambda_{\Lambda})\in \mathcal{Y}_{\{1\}}(R)$ and for each $\Lambda$ we write $\mathbb{D}_{\Lambda}$ for the evaluation of the {\Dieu} crystal of $X_{\Lambda}$. The action of $\calO_{D}$ on $X_{\Lambda}$ induces a decomposition $\DD_{\Lambda}=\DD_{\Lambda, 0}\oplus \DD_{\Lambda, 1}$. Since we have the isogenies $$X_{\Pi\Lambda_{0}}\rightarrow X_{\Lambda_{-1}}\rightarrow X_{\Lambda_{0}}\rightarrow X_{\Lambda_{1}}\rightarrow X_{\Pi^{-1}\Lambda_{0}},$$ it induces the following inclusions
\begin{equation*}
\begin{split}
&\DD_{\Pi\Lambda_{0}, 0}\subset \DD_{\Lambda_{-1}, 0}\subset \DD_{\Lambda_{0},0}\\
&\DD_{\Lambda_{0}, 0}\subset \DD_{\Lambda_{1}, 0}\subset \DD_{\Pi^{-1}\Lambda_{0},0}.\\
\end{split}
\end{equation*}
Then sending $y$ to $( \DD_{\Lambda_{-1}, 0}/ \DD_{\Pi\Lambda_{0},0}\subset \DD_{\Lambda_{0},0}/\DD_{\Pi\Lambda_{0}, 0},   \DD_{\Lambda_{1}, 0}/\DD_{\Lambda_{0}, 0} \subset \DD_{\Pi^{-1}\Lambda_{0},0}/\DD_{\Lambda_{0}, 0})\in \mathcal{Y}^{\prime}_{\{1\}}$ gives the desired map $\alpha_{\{1\}}$. Since $\calM_{\{1\}}$ is a discrete set of points, to prove the theorem we can restrict $\alpha_{\{1\}}$ to each fiber $\pi_{\{1\}}^{-1}(x)$ with $x\in\calM_{\{1\}}$. It is clear that $\pi_{\{1\}}$ is proper and hence $\pi_{\{1\}}^{-1}(x)$ is proper.  Therefore $\alpha_{\{1\}}$ restricted to $\pi^{-1}_{\{1\}}(x)$ is a proper morphism between integral schemes. Moreover Lemma \ref{fiber1} and Remark \ref{windremark} imply that $\alpha_{\{1\}}$ is bijective on points and the target is normal. Therefore $\alpha_{\{1\}}$ is an isomorphism when restricted to $\pi^{-1}(x)$ by Zariski's main theorem. Hence $\alpha_{\{1\}}$ itself is an isomorphism. 
\end{proof}

\subsection{The description of $\mathcal{Y}_{\{2\}}$} Let $x\in \calM^{\circ}_{\{2\}}(\FF)$ be a $\FF$-point. We would like to first understand the set  $\pi^{-1}_{\{1\}}(x)$.

\begin{lemma}\label{fiber2}
There is a bijection between $\pi^{-1}_{\{1\}}(x)$ and $\PP^{1}(\FF)$.
\end{lemma}

\begin{proof}
Again, we need to find all possible $S_{0}$ and $T_{0}$ such that they fit in the following web when $M_{0}$ is fixed

\begin{equation}
\begin{array}[c]{ccccccccc}
\tau(M_{0})&{\subset^{1}}&\tau(T^{\vee}_{0})&{\subset^{1}}&\tau(M^{\vee}_{0})&{\subset^{1}}&\tau(S^{\vee}_{0})&{\subset^{1}}&\frac{1}{p}\tau(M_{0})\\
\rotatebox{90}{$\subset$}&&\rotatebox{90}{$\subset$}&&\rotatebox{90}{$\subset$}&&\rotatebox{90}{$\subset$} &&\rotatebox{90}{$\subset$}\\
pM^{\vee}_{0}&{\subset^{1}}&S_{0}&{\subset^{1}} &M_{0} &{\subset^{1}} &T_{0} &{\subset^{1}}& M^{\vee}_{0}\\
\rotatebox{90}{$\subset$}&&\rotatebox{90}{$\subset$}&&\rotatebox{90}{$\subset$}&&\rotatebox{90}{$\subset$} &&\rotatebox{90}{$\subset$}\\
p\tau(M_{0})&{\subset^{1}}&p\tau(T^{\vee}_{0})&{\subset^{1}}&p\tau(M^{\vee}_{0})&{\subset^{1}}&p\tau(S^{\vee}_{0})&{\subset^{1}}&\tau(M_{0}).
\end{array} 
\end{equation}

We first note that as $M_{0}\in  \calM^{\circ}_{\{2\}}(\FF)$, it follows that $M_{0}\cap\tau(M_{0})$ is $\tau$-invariant but $M_{0}+ \tau(M_{0})$ is not $\tau$-invariant by Lemma \ref{BT-description}. Therefore $M^{\vee}_{0}+\tau(M^{\vee}_{0})$ is also $\tau$-invariant but $M^{\vee}_{0}\cap\tau(M^{\vee}_{0})$ is not $\tau$-invariant.  Given $y\in \pi^{-1}(x)$, it can be represented by a pair $(S_{0}, T_{0})$.  We define a map from $\pi^{-1}(x)\rightarrow \PP^{1}(M^{\vee}_{0}/M_{0})(\FF)$ by sending $(S_{0}, T_{0})$ to $T_{0}/M_{0}\subset M^{\vee}_{0}/M_{0}$. 

Conversely, let $\bar{T}_{0}$ be a point on $\PP^{1}(M^{\vee}_{0}/M_{0})(\FF)$ and let $T_{0}$ be the preimage of $\bar{T}_{0}$ under the natural reduction map $M^{\vee}_{0}\rightarrow M^{\vee}_{0}/M_{0}$. We will define a $S_{0}$ that is uniquely determined by $T_{0}$.  Here we need to distinguish three cases:
\begin{enumerate}
\item The case $\tau(M_{0})\not\subset T_{0}$ and $M_{0}\not\subset \tau(T^{\vee}_{0})$, 
\item The case $\tau(M_{0})\subset T_{0}$, 
\item The case $M_{0}\subset \tau(T^{\vee}_{0})$.
 \end{enumerate}

\emph{The case when $\tau(M_{0})\not\subset T_{0}$ and $M_{0}\not\subset \tau(T^{\vee}_{0})$}: We define $S^{\prime}_{0}=\tau( T^{\vee}_{0})\cap M_{0}$. Then we first verify that $pM^{\vee}_{0}\subset^{1} S^{\prime}_{0}\subset^{1} M_{0}$. Indeed, the inclusions are clear and we need to prove the index is $1$. Suppose that $S^{\prime}_{0}=pM^{\vee}_{0}$, then we  derive a contradiction. Consider the inclusion $$M_{0}\cap\tau(M_{0})\subset S^{\prime}_{0}=pM^{\vee}_{0}\subset^{2} M_{0}.$$
It follows that $\dim_{\FF}M_{0}/M_{0}\cap \tau(M_{0})\geq 2$. This is impossible by Lemma \ref{spin}.  Next we exclude the possibility $S^{\prime}_{0}=M_{0}$. Suppose otherwise $S^{\prime}_{0}=\tau( T^{\vee}_{0})\cap M_{0}=M_{0}$, then $M_{0}\subset \tau(T^{\vee}_{0})$ which is a contradiction to the assumption. Moreover we see $p\tau(T^{\vee}_{0})\subset S^{\prime}_{0} \subset  \tau(T^{\vee}_{0})$ without any difficulty. 
Next we define $p\tau(S^{\vee}_{0})=T_{0}\cap \tau(M_{0})$. We verify that $p\tau(M^{\vee}_{0})\subset^{1}p\tau(S^{\vee}_{0})\subset^{1}\tau(M_{0})$. Again the inclusions are obvious and we need to exclude the cases $p\tau(M^{\vee}_{0})=p\tau(S^{\vee}_{0})$ and $p\tau(S^{\vee}_{0})=\tau(M_{0})$. Suppose that $p\tau(M^{\vee}_{0})=p\tau(S^{\vee}_{0})$, then we have $M_{0}\cap \tau(M_{0})\subset T\cap \tau(M_{0})=p\tau(S^{\vee}_{0})=p\tau(M^{\vee}_{0})\subset^{2}M_{0}$ which is a contradiction to $\dim_{\FF}(M_{0}+\tau(M_{0}))/M_{0}\leq 1$. Now we assume that $p\tau(S^{\vee}_{0})=\tau(M_{0})$. This implies that $p\tau(S^{\vee}_{0})=T_{0}\cap \tau(M_{0})=\tau(M_{0})$ which in turn implies that $T_{0}\supset \tau(M_{0})$. This is a contradiction to the assumption. We are left to verify that $S_{0}=S^{\prime}_{0}$. To do this, we prove that $p\tau(S^{\prime\vee}_{0})=p\tau^{2}(T_{0})+p\tau(M^{\vee}_{0})\subset p\tau(S^{\vee}_{0})=T_{0}\cap \tau(M_{0})$. It is clear that $p\tau(M^{\vee}_{0})\subset T_{0}\cap \tau(M_{0})$ and we need to verify that $p\tau^{2}(T_{0})\subset T_{0}\cap \tau(M_{0})$. This follows from the following inclusions
$$p\tau^{2}(T_{0})\subset p\tau^{2}(M^{\vee}_{0})\subset p\tau(M^{\vee}_{0})+p\tau^{2}(M^{\vee}_{0})=pM^{\vee}_{0}+p\tau(M^{\vee}_{0})\subset T_{0}\cap \tau(M_{0}).$$
Since we have both $p\tau(M^{\vee}_{0})\subset^{1}p\tau(S^{\vee}_{0})\subset^{1}\tau(M_{0})$ and $p\tau(M^{\vee}_{0})\subset^{1}p\tau(S^{\prime\vee}_{0})\subset^{1}\tau(M_{0})$, we conclude that $S_{0}=S^{\prime}_{0}$. 

\emph{The case when $\tau(M_{0})\subset T_{0}$}: Since $M_{0}\neq \tau(M_{0})$, $T_{0}=M_{0}+\tau(M_{0})$ by an index consideration. 
In this case, we define $S_{0}=pM^{\vee}_{0}+p\tau(T^{\vee}_{0})$. We first verify that $pM^{\vee}_{0}\subset^{1}S_{0}\subset^{1} M^{\vee}_{0}$. Again the inclusions are clear and we only need to exclude that cases: $S_{0}=pM^{\vee}_{0}$ and $S_{0}=M^{\vee}_{0}$. Suppose that $S_{0}=pM^{\vee}_{0}$, then $pM^{\vee}_{0}+p\tau(T^{\vee}_{0})=pM^{\vee}_{0}$. This implies that $p\tau(T^{\vee}_{0})\subset pM^{\vee}_{0}$ which is equivalent to $M_{0}\subset \tau(T_{0})$. By assumption, this is equivalent to $M_{0}\subset \tau(M_{0}+\tau(M_{0}))$. But this implies that $M_{0}+\tau(M_{0})=\tau(M_{0}+\tau(M_{0})$ which is a contradiction to the assumption that $M_{0}+\tau(M_{0})$ is not $\tau$-stable. Suppose next that $S_{0}=M_{0}$, then $pM^{\vee}_{0}+p\tau(T^{\vee}_{0})=M_{0}$. It follows from this that $$pM^{\vee}_{0}+p\tau(M^{\vee}_{0})\supset pM^{\vee}_{0}+p\tau(T^{\vee}_{0})=M_{0}\supset^{2} pM^{\vee}_{0}.$$ As we have seen, this is a contradiction to $\dim_{\FF}M^{\vee}_{0}+\tau(M^{\vee}_{0})/M^{\vee}_{0}\leq 1$. We still need to verify that $p\tau(T^{\vee}_{0})\subset S_{0}\subset \tau(T^{\vee}_{0})$ and $p\tau(S^{\vee}_{0})\subset T_{0}\subset \tau(S^{\vee}_{0})$. The first one is clear. For the second one,  we have $p\tau(S^{\vee}_{0})=\tau(M_{0}\cap \tau(T_{0}))\subset M_{0}+\tau(M_{0})=T_{0}$. Since $T_{0}\subset \frac{1}{p}\tau(M_{0})$ and $T_{0}\subset \frac{1}{p}\tau^{2}(T_{0})$ by definition of $\tau$, we have $T_{0}\subset  \frac{1}{p}\tau(M_{0})\cap\frac{1}{p}\tau^{2}(T_{0})=\tau(S^{\vee}_{0})$.

\emph{The case when $M_{0}\subset \tau(T^{\vee}_{0})$}: Since $\tau(M_{0})\neq M_{0}$, we have $\tau(T^{\vee}_{0})=M_{0}+\tau(M_{0})$ by an index consideration. In this case, we let $\tau(S^{\vee}_{0})=\tau(M^{\vee}_{0})+T_{0}$. We first prove that $$\tau(M^{\vee}_{0})\subset^{1} \tau(S^{\vee}_{0})\subset^{1} \frac{1}{p}\tau(M_{0}).$$
Again the inclusions are clear and we verify the index is as indicated above. Suppose otherwise that $\tau(S^{\vee}_{0})=\tau(M^{\vee}_{0})+T_{0}=\tau(M^{\vee}_{0})$. This is equivalent to $T_{0}\subset \tau(M^{\vee}_{0})$ and is further equivalent to $\tau(T^{\vee}_{0})=M_{0}+\tau(M_{0})\supset \tau^{2}(M_{0})$. This implies that $\tau(M_{0}+\tau(M_{0}))=M_{0}+\tau(M_{0})$ which is a contradiction to our assumption. Next we assume that $\tau(S^{\vee}_{0})=\frac{1}{p}\tau(M_{0})$. This implies that $$p\tau(M^{\vee}_{0})\subset^{2} \tau(M_{0})=p\tau(S^{\vee}_{0})=pT_{0}+p\tau(M^{\vee}_{0})\subset pM^{\vee}_{0}+ p\tau(M^{\vee}_{0}).$$
We have seen many times this is impossible. The inclusion $p\tau(S^{\vee}_{0})\subset T_{0}\subset \tau(S^{\vee}_{0})$ is clear and we are left to show that $p\tau(T^{\vee}_{0})\subset S_{0}\subset \tau(T^{\vee}_{0})$. Notice that $\tau(S_{0})=T^{\vee}_{0}\cap\tau( M_{0})$ and hence that $\tau(S_{0})=\tau(M_{0})\cap T^{\vee}_{0}\subset \tau^{2}(T^{\vee}_{0})=\tau(M_{0})+\tau^{2}(M_{0})$. Therefore $S_{0}\subset \tau(T^{\vee}_{0})$. On the other hand, $p\tau^{2}(T^{\vee}_{0})\subset p\tau(M_{0})$ is clear and $p\tau^{2}(T^{\vee}_{0})\subset T^{\vee}_{0}$ by definition. Therefore $p\tau^{2}(T^{\vee}_{0})\subset  p\tau(M_{0})\cap T^{\vee}_{0}=\tau(S_{0})$ and we have $p\tau(T^{\vee}_{0})\subset S_{0}\subset \tau(T^{\vee}_{0})$.
\end{proof}

Now we describe the global structure of $\mathcal{Y}_{\{2\}}$. Let $L_{2}$ be a vertex lattice of type $2$. We will describe the scheme $\mathcal{Y}_{L_{2}}=\pi^{-1}_{\{1\}}(\calM^{\circ}_{L_{2}})$. 

\begin{theorem}\label{YL2}
The scheme $\mathcal{Y}_{L_{2}}$ is irreducible and is a $\PP^{1}$-bundle over $\calM^{\circ}_{L_{2}}$. In particular its dimension is $3$. Its closure $\overline{\mathcal{Y}}_{L_{2}}$ is a $\PP^{1}$-bundle over $\calM_{L_{2}}$ and is an irreducible component of $\calM_{\Iw}$.\end{theorem}

\begin{proof}
Let $R$ be a $\FF$-algebra, then a point $y\in \mathcal{Y}_{L_{2}}(R)$ is represented by a quadruple $y=(X_{\Lambda}, \iota_{\Lambda}, \rho_{\Lambda}, \lambda_{\Lambda})$ that fits in the following two isogenies
\begin{equation}
\begin{split}
& \mathbb{X}_{L^{+}_{2}}\rightarrow X_{\Lambda_{0}} \rightarrow  \mathbb{X}_{L^{-}_{2}}\\
& X_{\Lambda_{0}}\rightarrow X_{\Lambda_{1}}\rightarrow X_{\Pi^{-1}\Lambda_{0}}\\
\end{split}
\end{equation}

Consider the $\PP(\DD_{\Pi^{-1}\Lambda_{0}, 0}/\DD_{\Lambda_{0},0})$-bundle over $\calM^{\circ}_{L_{2}}$ which we denote it by $\mathcal{Y}^{\prime}_{L_{2}}$. Therefore we can define a map $\alpha_{\{2\}}: \mathcal{Y}_{L_{2}}\rightarrow  \mathcal{Y}^{\prime}_{L_{2}}$ by sending $y=(X_{\Lambda}, \iota_{\Lambda}, \rho_{\Lambda}, \lambda_{\Lambda})$ to 
\begin{equation*}
(\pi_{\{1\}}(y), \DD_{\Lambda_{1},0}/\DD_{\Lambda_{0},0}\subset\DD_{\Pi^{-1}\Lambda_{0},0}/\DD_{\Lambda_{0},0}).
\end{equation*}

By the proof of Lemma \ref{fiber2} and Remark \ref{windremark}, we see $\mathcal{Y}_{ L_{2}}(k)=\mathcal{Y}^{\prime}_{L_{2}}(k)$ for any field extension $k$ over $\FF$. 
The morphism $\alpha_{\{2\}}$ is a proper morphism as both  $\pi_{\{1\}}: \mathcal{Y}_{L_{2}}\rightarrow \calM^{\circ}_{L_{2}}$ and the natural map $\mathcal{Y}^{\prime}_{L_{2}}\rightarrow \calM^{\circ}_{L_{2}}$ are proper.  Since $\mathcal{Y}^{\prime}_{L_{2}}$ is also normal, one can apply Zariski's main theorem to conclude that $\alpha_{\{2\}}$ is an isomorphism. It is clear that the closure $\overline{\mathcal{Y}}_{L_{2}}$ is a $\PP^{1}$-bundle over $\calM_{L_{2}}$. 
\end{proof}

\subsection{The description of $\mathcal{Y}_{\{0\}}$} Let $x\in \calM^{\circ}_{\{0\}}(\FF)$ ba a $\FF$-point. We would like to first understand the set  $\pi^{-1}_{\{1\}}(x)$. This case is symmetric to the case of $\mathcal{Y}_{\{2\}}$. Therefore we will only point out the difference between these two cases and leave the details to the reader. 

\begin{lemma}\label{fiber0}
There is a bijection between $\pi_{\{1\}}^{-1}(x)$ and $\PP^{1}(\FF)$.
\end{lemma}

\begin{proof}
Let $M_{0}\in \calM_{\{0\}}$, then it follows from Lemma \ref{BT-description} that $M_{0}+\tau(M_{0})$ is $\tau$-invariant but $M_{0}\cap \tau(M_{0})$ is not $\tau$-invariant. This is equivalent to $M^{\vee}_{0}\cap \tau(M^{\vee}_{0})$ is $\tau$-invariant but $M^{\vee}_{0}+\tau(M^{\vee}_{0})$ is not $\tau$-invariant. Given a point $y\in\pi^{-1}(x)$, we represent it by $(S_{0}, T_{0})$ which fits in the diagram \eqref{web}. Then we define a map $\pi^{-1}(x)\rightarrow \PP^{1}(M_{0}/pM^{\vee})(\FF)$ by sending $(S_{0}, T_{0})$ to $S_{0}/pM^{\vee}_{0}\subset M_{0}/pM^{\vee}_{0}$. We are naturally divided into three cases: the case $S_{0}\not\supset p\tau(M^{\vee}_{0})$ and $\tau(S^{\vee}_{0})\not\supset M^{\vee}_{0}$, the case $S_{0}\supset p\tau(M^{\vee}_{0})$ and the case $\tau(S^{\vee}_{0})\supset M^{\vee}_{0}$. 

\emph{The case $S_{0}\not\supset p\tau(M^{\vee}_{0})$ and $\tau(S^{\vee}_{0})\not\supset M^{\vee}_{0}$:} This is completely symmetric to the proof in the corresponding case in Lemma \ref{fiber2}. We skip all the details.

\emph{The case $S_{0}\supset p\tau(M^{\vee}_{0})$:}  We point out how the condition that $M_{0}\cap \tau(M_{0})$ is not $\tau$-invariant comes into the proof. In this case we will define $T_{0}=M_{0}+p\tau(S^{\vee}_{0})$.  Then we need to check that $M_{0}\subset^{1} T_{0}\subset^{1} M^{\vee}_{0}$. Suppose that $M_{0}=T_{0}=M_{0}+p\tau(S^{\vee}_{0})$ then $M_{0}\supset p\tau(S^{\vee}_{0})$. Note that $S_{0}=pM^{\vee}_{0}+p\tau(M^{\vee}_{0})$ in this case by an index consideration. It follows that $M_{0}\supset p\tau(S^{\vee}_{0})=\tau(M_{0})\cap\tau^{2}(M_{0})$. 
This implies that $M_{0}\cap\tau(M_{0})\supset \tau(M_{0})\cap\tau^{2}(M_{0})$ which in turn implies that $M\cap\tau(M_{0})$ is $\tau$-invariant which is contradiction. We leave the rest of the proof for this Lemma to the reader.

\end{proof}

Now we describe the global structure of $\mathcal{Y}_{\{0\}}$. Let $L_{0}$ be a vertex lattice of type $0$. We will describe the scheme $\mathcal{Y}_{L_{0}}=\pi^{-1}_{\{1\}}(\calM^{\circ}_{L_{0}})$. 

\begin{theorem}\label{YL0}
The scheme $\mathcal{Y}_{L_{0}}$ is irreducible and is a $\PP^{1}$-bundle over $\calM^{\circ}_{L_{0}}$. In particular its dimension is $3$. Its closure $\overline{\mathcal{Y}}_{L_{0}}$ is a $\PP^{1}$-bundle over $\calM_{L_{0}}$ and is an irreducible component of $\calM_{\Iw}$.
\end{theorem}
\begin{proof}
This follows from the same proof as in Theorem \ref{YL2} with simple modifications.
\end{proof}

\subsection{The description of $\mathcal{Y}_{\{02\}}$} Let $x\in \calM^{\circ}_{\{02\}}(\FF)$ be a $\FF$-point. We now describe the set  $\pi_{\{1\}}^{-1}(x)$. We define $E$ to be the transversal intersection of two copies of $\PP^{1}$ at a single point. 
\begin{lemma}\label{fiber02}
There is a bijection between $\pi^{-1}_{\{1\}}(x)$ and $E(\FF)$.
\end{lemma}

\begin{proof}
For the most part, the proof is the same as the proof for Lemma \ref{fiber2}. First we note that as $M_{0}\in  \calM^{\circ}_{\{02\}}(\FF)$, it follows that both $M_{0}\cap\tau(M_{0})$ and  $M_{0}+ \tau(M_{0})$ are $\tau$-invariant by Lemma \ref{BT-description}. Therefore $M^{\vee}_{0}+\tau(M^{\vee}_{0})$ and $M^{\vee}_{0}\cap\tau(M^{\vee}_{0})$ are also $\tau$-invariant. However, $M_{0}$ is not $\tau$-invariant. Let $y\in \pi^{-1}(x)$, then it can be represented by $(S_{0}, T_{0})$ that fits in the diagram \eqref{web}. We define a map $\theta: \pi^{-1}(x)\rightarrow \PP^{1}(M_{0}/pM^{\vee}_{0})(\FF)\times \PP^{1}(M^{\vee}_{0}/M_{0})(\FF)$ by sending $y=(S_{0}, T_{0})$ to $\bar{S}_{0}=S_{0}/pM^{\vee}_{0}\subset M_{0}/pM^{\vee}_{0}$ and $\bar{T}_{0}=T_{0}/M_{0}\subset M^{\vee}_{0}/M_{0}$. Let $\PP_{1}$ and $\PP_{2}$ be the two projective lines that are contained in $E$ and let $P$ be their unique intersection point. We embed $E(\FF)$ in $\PP^{1}(M_{0}/pM^{\vee}_{0})(\FF)\times \PP^{1}(M^{\vee}_{0}/M_{0})(\FF)$ in such a way that $\PP_{1}(\FF)$ is identified with $\PP^{1}(M_{0}/pM^{\vee}_{0})(\FF)$ with $P$ identified with the point $M_{0}\cap\tau(M_{0})/pM^{\vee}_{0}\subset M_{0}/pM^{\vee}_{0}$ and $\PP_{2}(\FF)$ identified with $\PP^{1}(M_{0}/pM^{\vee}_{0})(\FF)$ with $P$ identified with the point $M_{0}+\tau(M_{0})/M_{0}\subset M^{\vee}_{0}/M_{0}$. First we prove that the image of $\mathcal{Y}_{\{02\}}\rightarrow \PP^{1}(M_{0}/pM^{\vee}_{0})(\FF)\times \PP^{1}(M^{\vee}_{0}/M_{0})(\FF)$ lands in $E(\FF)$. Indeed consider the three cases studied in Lemma \ref{fiber2}:
\begin{enumerate} 
\item The case $\tau(M_{0})\not\subset T_{0}$ and $M_{0}\not\subset \tau(T^{\vee}_{0})$;
\item The case $\tau(M_{0})\subset T_{0}$;
\item The case $M_{0}\subset \tau(T^{\vee}_{0})$. 
\end{enumerate}
In the case when $\tau(M_{0})\not\subset T_{0}$ and $M_{0}\not\subset \tau(T^{\vee}_{0})$,  we have seen that $S_{0}$ is uniquely determined by $T_{0}$. Therefore the map $\theta$ when restricted to those $(S_{0}, T_{0})$ in this case lands in $L_{2}$ away from the point defined by $M_{0}+\tau(M_{0})$. Next we consider the case when either  $\tau(M_{0})\subset T_{0}$ or $M_{0}\subset \tau(T^{\vee}_{0})$. In fact these two cases are equivalent. Indeed suppose $\tau(M_{0})\subset T_{0}$, then $T_{0}=M_{0}+\tau(M_{0})$.  Therefore $\tau(T^{\vee}_{0})=T^{\vee}_{0}=M^{\vee}_{0}\cap \tau(M^{\vee}_{0})$ which also contains $M_{0}$. Note that in this case we have $\tau(T^{\vee}_{0})=M_{0}+\tau(M_{0})$ as well. One can show if $M_{0}\subset \tau(T^{\vee}_{0})$, then $\tau(M_{0})\subset T_{0}$ in a similar way. Thus when $T_{0}=M_{0}+\tau(M_{0})$, sending $S_{0}$ to $\bar{S}_{0}=S_{0}/pM^{\vee}_{0}$ gives a point in $\PP^{1}(M_{0}/pM^{\vee}_{0})$. We see that the image of $\theta$ is indeed in $E(\FF)$. 

Conversely, let $Q=(\bar{S}_{0}, \bar{T}_{0})\in E(\FF)$ be any point. Here $E(\FF)$ is identified as a subset of $\pi^{-1}(x)\rightarrow \PP^{1}(M_{0}/pM^{\vee}_{0})(\FF)\times \PP^{1}(M^{\vee}_{0}/M_{0})(\FF)$. We write $(S_{0}, T_{0})$ the preimage of $(\bar{S}_{0}, \bar{T}_{0})$ under the natural reduction map. Suppose first that $Q\in \PP_{2}-P$, then we have seen that for each $T_{0}$ we can find a unique $S_{0}$ such that $(S_{0}, T_{0})\in \pi^{-1}(x)$. Then we can assume that $T_{0}=M_{0}+\tau(M_{0})$ and hence $\tau(T^{\vee}_{0})=M_{0}+\tau(M_{0})$. Let $\bar{S}_{0}\in \PP^{1}(M_{0}/pM^{\vee}_{0})(\FF)$ be any point. Then we have $pM^{\vee}_{0}\subset S_{0}\subset M_{0}$. It follows immediately that $S_{0}\subset M_{0}\subset M_{0}+\tau(M_{0})=\tau(T^{\vee}_{0})$ and $p\tau(T^{\vee}_{0})=pM_{0}+p\tau(M_{0})\subset pM^{\vee}_{0}\subset S_{0}$. Therefore $(S_{0}, T_{0})$ gives a point in $\pi^{-1}(x)$ in this case. This finishes the proof as in this case $\bar{S}_{0}$ can be chosen to be any point on the line $\PP^{1}(M_{0}/pM^{\vee}_{0})(\FF)$. 
\end{proof}

\begin{theorem}\label{YL02}
Let $\mathcal{Y}_{L_{0}, L_{2}}=\pi^{-1}_{\{1\}}(\calM^{\circ}_{L_{0}, L_{2}})$ and suppose that it is non-empty. 
Then $\mathcal{Y}_{L_{0}, L_{2}}$ has two irreducible components $\mathcal{Y}^{T_{0}}_{L_{0},L_{2}}$ and $\mathcal{Y}^{S_{0}}_{L_{0}, L_{2}}$. Their intersection is a projective line with its $\FF_{p}$-points removed. The closure of $\mathcal{Y}^{T_{0}}_{L_{0},L_{2}}$ is a ruled surface. The same is true for $\mathcal{Y}^{S_{0}}_{L_{0},L_{2}}$. 
\end{theorem}
\begin{proof}
Let $R$ be a $\FF$-algebra, then an $R$-point $y\in \mathcal{Y}_{L_{0}, L_{2}}$ is represented by a tuple 
\begin{equation}
y=\underline{X}_{\Lambda}=(X_{\Lambda}, \iota_{\Lambda}, \rho_{\Lambda}, \lambda_{\Lambda})
\end{equation}
that fits in the following chain of isogenies
\begin{equation}
\begin{split}
&\mathbb{X}_{L^{+}_{2}} \rightarrow X_{\Lambda_{0}} \rightarrow \mathbb{X}_{L^{+}_{0}}\\
& X_{\Pi\Lambda_{0}}\rightarrow X_{\Lambda_{-1}} \rightarrow X_{\Lambda_{0}}\\ 
\end{split}
\end{equation}
and 
\begin{equation}
\begin{split}
& \mathbb{X}_{L^{+}_{2}} \rightarrow X_{\Lambda_{0}} \rightarrow \mathbb{X}_{L^{+}_{0}}\\
& X_{\Lambda_{0}} \rightarrow X_{\Lambda_{1}}\rightarrow X_{\Pi^{-1}\Lambda_{0}}\\ 
\end{split}
\end{equation}

Consider the $\PP(\DD_{\Lambda_{0}, 0}/\DD_{\Pi\Lambda_{0},0})\times \PP(\DD_{\Pi^{-1}\Lambda_{0},0}/\DD_{\Lambda_{0},0})$-bundle over $\calM^{\circ}_{L_{0}, L_{2}}$ which we denote it by $\mathcal{Y}^{\prime}_{L_{0}, L_{2}}$. Therefore we can define a map $\alpha_{\{02\}}: \mathcal{Y}_{L_{0}, L_{2}}\rightarrow  \mathcal{Y}^{\prime}_{L_{0}, L_{2}}$ by sending $y=(X_{\Lambda}, \iota_{\Lambda}, \rho_{\Lambda}, \lambda_{\Lambda})$ to 
\begin{equation*}
\begin{split}
&(\pi_{\{1\}}(y), \DD_{\Lambda_{-1},0}/\DD_{\Pi\Lambda_{0},0}\subset\DD_{\Lambda_{0},0}/\DD_{\Pi\Lambda_{0}, 0},\\ &\DD_{\Lambda_{1},0}/\DD_{\Lambda_{0},0}\subset\DD_{\Pi^{-1}\Lambda_{0},0}/\DD_{\Lambda_{0},0}).\\
\end{split}
\end{equation*}

Define the closed subschemes of $\mathcal{Y}_{L_{0}, L_{2}}$ by
\begin{equation}\label{YST}
\begin{split}
&\mathcal{Y}^{S_{0}}_{L_{0}, L_{2}}=\{y=\underline{X}_{\Lambda}\in \mathcal{Y}_{L_{0},L_{2}}:  \DD_{\Lambda_{1},0}=\DD_{\Lambda_{0},0}+\Pi^{-1}F\DD_{\Lambda_{0},0}\}\\
&\mathcal{Y}^{T_{0}}_{L_{0}, L_{2}}=\{y=\underline{X}_{\Lambda}\in \mathcal{Y}_{L_{0},L_{2}}:  \DD_{\Lambda_{-1},0}=\DD_{\Lambda_{0},0}\cap\Pi^{-1}F\DD_{\Lambda_{0},0}\}.\\
\end{split}
\end{equation}

Notice that 
\begin{equation}
\begin{split}
&\mathcal{Y}^{S_{0}}_{L_{0}, L_{2}}(\FF)=\{(M_{0}, S_{0}, T_{0})\in \mathcal{Y}_{L_{0}, L_{2}}(\FF): T_{0}=M_{0}+\tau(M_{0}) \}\\ 
&\mathcal{Y}^{T_{0}}_{L_{0}, L_{2}}(\FF)=\{(M_{0}, S_{0}, T_{0})\in \mathcal{Y}_{L_{0}, L_{2}}(\FF): S_{0}=M_{0}\cap\tau(M_{0}) \}.\\
\end{split}
\end{equation}
By the proof of Lemma \ref{fiber02}, we see $\mathcal{Y}_{L_{0}, L_{2}}(k)=\mathcal{Y}^{S_{0}}_{L_{0}, L_{2}}(k)\cup \mathcal{Y}^{T_{0}}_{L_{0}, L_{2}}(k)$ for any field extension $k$ of $\FF$. Therefore $\mathcal{Y}_{L_{0}, L_{2}}=\mathcal{Y}^{S_{0}}_{L_{0}, L_{2}}\cup \mathcal{Y}^{T_{0}}_{L_{0}, L_{2}}$.
Let $\mathcal{Y}^{\prime S_{0}}_{L_{0}, L_{2}}$ be the $\PP(\DD_{\Lambda_{0}, 0}/\DD_{\Pi\Lambda_{0},0})$-bundle  over $\calM^{\circ}_{L_{0}, L_{2}}$ in $\mathcal{Y}^{\prime}_{L_{0}, L_{2}}$ defined by those points
\begin{equation*}
\begin{split}
&(\pi_{\{1\}}(y), \DD_{\Lambda_{-1},0}/\DD_{\Pi\Lambda_{0},0}\subset\DD_{\Lambda_{0},0}/\DD_{\Pi\Lambda_{0}, 0},\\
&(\DD_{\Lambda_{0},0}+\Pi^{-1}F\DD_{\Lambda_{0},0})/\DD_{\Lambda_{0},0}\subset\DD_{\Pi^{-1}\Lambda_{0},0}/\DD_{\Lambda_{0},0}).\\
\end{split}
\end{equation*}
We see that the restriction of $\alpha_{\{02\}}$ to $\mathcal{Y}^{S_{0}}_{L_{0}, L_{2}}$ defines a morphism 
\begin{equation*}
\mathcal{Y}^{S_{0}}_{L_{0}, L_{2}}\rightarrow  \mathcal{Y}^{\prime S_{0}}_{L_{0}, L_{2}}.
\end{equation*}
This morphism is easily seen to be a proper morphism and $\mathcal{Y}^{\prime S_{0}}_{L_{0}, L_{2}}$ is normal. Moreover by the proof of Lemma \ref{fiber02} this map is bijective on points. Thus by Zariski's main theorem this map is an isomorphism. Let $\mathcal{Y}^{\prime T_{0}}_{L_{0}, L_{2}}$ be the $\PP(\DD_{\Pi^{-1}\Lambda_{0}, 0}/\DD_{\Lambda_{0},0})$-bundle  over $\calM^{\circ}_{L_{0}, L_{2}}$ in $\mathcal{Y}^{\prime}_{L_{0}, L_{2}}$ defined by those points
\begin{equation*}
\begin{split}
&(\pi_{\{1\}}(y), (\DD_{\Lambda_{0},0}\cap \Pi^{-1}F \DD_{\Lambda_{0},0}) /\DD_{\Pi\Lambda_{0},0}\subset\DD_{\Lambda_{0},0}/\DD_{\Pi\Lambda_{0}, 0},\\ 
&\DD_{\Lambda_{1},0}/\DD_{\Lambda_{0},0}\subset\DD_{\Pi^{-1}\Lambda_{0},0}/\DD_{\Lambda_{0},0}).\\
\end{split}
\end{equation*}
We see that the restriction of $\alpha_{\{02\}}$ to $\mathcal{Y}^{T_{0}}_{L_{0}, L_{2}}$ defines a morphism 
\begin{equation*}
\mathcal{Y}^{T_{0}}_{L_{0}, L_{2}}\rightarrow  \mathcal{Y}^{\prime T_{0}}_{L_{0}, L_{2}} 
\end{equation*}
which is an isomorphism by the same argument as before. In particular, $\mathcal{Y}^{T_{0}}_{L_{0}, L_{2}}$ and $\mathcal{Y}^{S_{0}}_{L_{0}, L_{2}}$ are irreducible. The rest of the claims are clear.

\end{proof}

\subsection{Summary and an application}We summarize the results obtained so far in the following theorems and conclude our study of the scheme $\calM_{\Iw}$. As a consequence, we apply the results here to describe the scheme $\calM_{P}$ via the correspondence in \eqref{corres}.

\begin{theorem}\label{main-thm}
The scheme $\calM_{\Iw}$ admits a decomposition of the form
$$\calM_{\Iw}=\mathcal{Y}_{\{0\}}\sqcup \mathcal{Y}_{\{2\}}\sqcup \mathcal{Y}_{\{02\}}\sqcup \mathcal{Y}_{\{1\}}.$$
\begin{enumerate}
\item The scheme $\mathcal{Y}_{\{0\}}$ is of dimension $3$ and an irreducible component of it is of the form $\mathcal{Y}_{L_{0}}$. The scheme $\mathcal{Y}_{L_{0}}$ is a $\PP^{1}$-bundle over $\calM^{\circ}_{L_{0}}$ and therefore its closure $\overline{\mathcal{Y}}_{L_{0}}$ is  a $\PP^{1}$-bundle over $\calM_{L_{0}}$.
\item The scheme $\mathcal{Y}_{\{2\}}$ is of dimension $3$ and an irreducible component of it is of the form $\mathcal{Y}_{L_{2}}$. The scheme $\mathcal{Y}_{L_{2}}$ is a $\PP^{1}$-bundle over $\calM^{\circ}_{L_{2}}$ and therefore its closure $\overline{\mathcal{Y}}_{L_{2}}$ is  a $\PP^{1}$-bundle over $\calM_{L_{2}}$.
\item The scheme $\mathcal{Y}_{\{02\}}$ is of dimension $2$ and an irreducible component of it is of the form $\mathcal{Y}^{T_{0}}_{L_{0}, L_{2}}$ or $\mathcal{Y}^{S_{0}}_{L_{0}, L_{2}}$. They are both $\PP^{1}$-bundle over $\calM^{\circ}_{L_{0},L_{2}}$.
\item The scheme $\mathcal{Y}_{\{1\}}$ is of dimension $2$ and an irreducible component of it is of the form $\mathcal{Y}_{L_{1}}$. The scheme $\mathcal{Y}_{L_{1}}$  is a  $\PP^{1}\times \PP^{1}$-bundle over $\calM_{L_{1}}$.
\end{enumerate}
\end{theorem} 

\begin{proof}
Part $(1)$ is proved in Theorems \ref{YL0}, $(2)$ is proved in Theorem \ref{YL2},  $(3)$ is the content of Theorem \ref{YL02} and $(4)$ is proved in \ref{YL1}.
\end{proof}

\begin{corollary}\label{main-cor}
The scheme $\calM_{\Iw}$ has three types of components. They are of the form $\overline{\mathcal{Y}}_{L_{0}}$, $\overline{\mathcal{Y}}_{L_{2}}$ and $\overline{\mathcal{Y}}_{L_{1}}$.
\begin{enumerate}
\item The intersection between $\overline{\mathcal{Y}}_{L_{0}}$ and $\overline{\mathcal{Y}}_{L_{2}}$ is $\mathcal{M}_{L_{0},L_{2}}$ which is isomorphic to $\PP^{1}$.
\item The intersection between $\overline{\mathcal{Y}}_{L_{0}}$ and ${\mathcal{Y}}_{L_{1}}$ is a $\PP^{1}$-bundle over $\calM_{L_{1}}$.
\item The intersection between $\overline{\mathcal{Y}}_{L_{2}}$ and ${\mathcal{Y}}_{L_{1}}$ is a $\PP^{1}$-bundle over $\calM_{L_{1}}$.
\end{enumerate}
\end{corollary}

\begin{proof}
The claims about the intersection can be checked on points. We recall that we have proved the following
\begin{equation*}
\begin{split}
& \mathcal{Y}_{L_{1}}=\{(M_{0}, S_{0}, T_{0}): M_{0}\in\calM_{L_{1}}, pM^{\vee}_{0}\subset S_{0}\subset M_{0}, M_{0}\subset T_{0}\subset M^{\vee}_{0}\};\\
& \mathcal{Y}_{L_{0}}=\{(M_{0}, S_{0}, T_{0}): M_{0}\in\calM^{\circ}_{L_{0}}, pM^{\vee}_{0}\subset S_{0}\subset M_{0},\text{$T_{0}$ is uniquely determined by $S_{0}$}\};\\
& \mathcal{Y}_{L_{2}}=\{(M_{0}, S_{0}, T_{0}): M_{0}\in\calM^{\circ}_{L_{2}}, M_{0}\subset T_{0}\subset M^{\vee}_{0},\text{$S_{0}$ is uniquely determined by $T_{0}$}\};\\
& \mathcal{Y}_{L_{0},L_{2}}=\{(M_{0}, S_{0}, T_{0}): M_{0}\in\calM^{\circ}_{L_{0}, L_{2}}, M_{0}\subset T_{0}\subset M^{\vee}_{0},pM^{\vee}_{0}\subset S_{0}\subset M_{0},\\ 
&\text{$T_{0}$ determines $S_{0}$ if $T_{0}\neq M_{0}+\tau(M_{0})$},\text{$S_{0}$ determines $T_{0}$ if $S_{0}\neq M_{0}\cap\tau(M_{0})$} \}.\\
\end{split}
\end{equation*}

From this all the claims about intersection can be proved. It follows that $\mathcal{Y}_{L_{1}}$ is not contained in $\overline{\mathcal{Y}}_{L_{0}}$ or $\overline{\mathcal{Y}}_{L_{2}}$. Moreover $\overline{\mathcal{Y}}_{L_{0}}$ and $\overline{\mathcal{Y}}_{L_{2}}$ are clearly irreducible components. Therefore $\mathcal{Y}_{L_{1}}$ is also an irreducible components. 
\end{proof}

Recall that we have the Rapoport-Zink space $\calN_{P}$ corresponding to the lattice chain $\calL_{\{02\}}$ and the integral Rapoport-Zink datum   $\mathcal{D}_{\{02\}}$. We have a natural map $\pi_{\{02\}}: \calN_{\Iw}\rightarrow \calN_{P}$ that corresponds to the natural map $\pi_{\{02\}}:\calL_{\emptyset}\rightarrow \calL_{\{02\}}$. We restrict $\pi_{\{02\}}$ to $\calM_{\Iw}$  and obtain $\pi_{\{02\}}: \calM_{\Iw}\rightarrow \calM_{\{02\}}$.  The following theorem is an immediate consequence of the results we obtained in Theorem \ref{main-thm} and Corollary \ref{main-cor}. 

\begin{theorem}\label{main-app}
The scheme $\calM_{P}$ is smooth and equidimensional of dimension $2$. In fact we have a decomposition
\begin{equation*}
\calM_{P}=\bigsqcup_{x\in \calM_{\{1\}}}(\PP^{1}\times\PP^{1})_{x}.
\end{equation*} 
\end{theorem}
\begin{proof}
By Theorem \ref{main-thm}, we have the following decomposition of $\calM_{\Iw}$
\begin{equation}
\calM_{\Iw}=(\cup_{L_{0}}\bar{\mathcal{Y}}_{L_{0}})\cup(\cup_{L_{2}}\bar{\mathcal{Y}}_{L_{2}})\cup(\cup_{L_{1}}{\mathcal{Y}}_{L_{1}}).
\end{equation}
For each vertex lattice $L$, the map $\pi_{\{02\}}$ will contract each $\calM_{L}$ to a point. Note that for $L=L_{0}$ a vertex lattice of type $0$ or $L=L_{2}$ a vertex lattice of type $2$, $\bar{\mathcal{Y}}_{L}$ is a $\PP^{1}$-bundle over $\calM_{L}$ hence its image under $\pi_{\{02\}}$ is a collection of $\PP^{1}$. Since the scheme $\calM_{\{1\}}=\bigcup_{L_{1}}\calM_{L_{1}}$ is discrete set of points,  $\bar{\mathcal{Y}}_{L_{1}}$ maps to $$\bigsqcup_{x\in \calM_{\{1\}}}(\PP^{1}\times\PP^{1})_{x}$$ under $\pi_{\{02\}}$. Since $\bar{\mathcal{Y}}_{L}\cap \bar{\mathcal{Y}}_{L_{1}}$ is a $\PP^{1}$-bundle over $\calM_{L_{1}}$ for any vertex lattice $L$ of type $0$ or type $2$, the image of  ${\mathcal{Y}}_{L}$ under $\pi_{\{02\}}$ is a $\PP^{1}$ contained in the image of ${\mathcal{Y}}_{L_{1}}$.
\end{proof}

\begin{remark}
The proof of this theorem is similar to that of \cite[Theorem 4.7]{Yu06}. The supersingular locus of the usual Siegel threefold with Iwahori level is described in \cite[Theorem 8.1]{Yu08}. The reader is invited to compare our results with his.
\end{remark}

\section{Affine Deligne-Lusztig varieties}

\subsection{Affine Deligne Lusztig variety} In this section we switch to a purely group theoretic setting and compare the previous results obtained by studying the Rapoport-Zink space in terms of lattices with the results of \cite{GHN16} obtained for those fully Hodge-Newton decomposable affine Deligne-Lusztig varieties. From here on we will abbreviate affine Deligne-Lusztig varieties as ADLV. 

Let $F$ be a finite extension of $\QQ_{p}$ and $\breve{F}$ be the completion of the maximal unramified extension of $F$. Let $G$ be a connected reductive group over $F$ which we assume to be  and we write $\breve{G}$ its base change to $\breve{F}$. Then $\breve{G}$ is quasi split and we choose a maximal split torus $S$ and denote by $T$ its centralizer.  We know $T$ is a maximal torus and we denote by $N$ its normalizer. The relative Weyl group is defined to be $W=N(\breve{F})/ T(\breve{F})$. This is a finite group. Let $\Gamma$ be the Galois group of $\breve{F}$ and we have the following Kottwitz homomorphism \cite{RR96}:
$$\kappa_{G}: G(\breve{F})\rightarrow X_{*}(\breve{G})_{\Gamma}.$$
Denote by $\widetilde{W}$ the \emph{Iwahori Weyl group} of $\breve{G}$ which is by definition $\widetilde{W}=N(\breve{F})/ T(\breve{F})_{1}$ where $T(\breve{F})_{1}$ is the kernel of the Kottwitz homomorphism for $T(\breve{F})$. Let $\breve{\mathfrak{B}}(G)$ be the Bruhat-Tits building of $G$ over $\breve{F}$. The choice of $S$ determines an standard apartment $\breve{\mathfrak{A}}$ which $\widetilde{W}$ acts on by affine transformations. We fix a $\sigma$-invariant alcove $\mathfrak{a}$ and a special vertex of $\mathfrak{a}$. Inside the Iwahori Weyl group $\widetilde{W}$, there is a copy of the affine Weyl group $W_{a}$ which can be identified with $N(\breve{F})\cap G(\breve{F})_{1}/ T(\breve{F})_{1}$ where $G(\breve{F})_{1}$ is the kernel of the Kottwitz morphism for $G(\breve{F})$. The group $\widetilde{W}$ is not quite a Coxeter group while $W_{a}$ is generated by the affine reflections denoted by $\tilde{\mathbb{S}}$ and $(\widetilde{W}, \tilde{\mathbb{S}})$ form a Coxeter system. We in fact have $\widetilde{W}= W_{a}\rtimes \Omega$ where $\Omega$ is the normalizer of a fixed base alcove $\mathfrak{a}$ and more canonically $\Omega=X_{*}(T)_{\Gamma}/ X_{*}(T_{sc})_{\Gamma}$ where $T_{sc}$ is the preimage of $T\cap G^{\der}$ in the simply connected cover $G_{sc}$ of $G^{\der}$.

Let $\mu\in X_{*}(T)$ be a minuscule cocharacter of $G$ over $\breve{F}$ and $\lambda$ its image in $X_{*}(T)_{I}$.  We denote by $\tau$ the projection of $\lambda$ in $\Omega$. The \emph{admissible subset} of $\widetilde{W}$ is defined to be
$$\Adm(\mu)=\{w\in \widetilde{W}; w \leq x(\lambda) \text{ for some }x\in W \}.$$
Here $\lambda$ is considered as a translation element in $\widetilde{W}$. Let $K\subset\tilde{\mathbb{S}}$ and $\breve{K}$ its corresponding parahoric subgroup. Here by our convention, if $K=\emptyset$, then $\breve{K}=\breve{I}$ is the Iwahori subgroup. Let $\widetilde{W}_{K}$ be the subgroup defined by $N(\breve{F})\cap \breve{K}/T(\breve{F})_{1}$. We have the identification $\breve{K}\backslash G(\breve{F})/\breve{K}= \widetilde{W}_{K}\backslash \widetilde{W}/ \widetilde{W}_{K}$. Therefore we can define a relative position map 
\begin{equation}
\begin{split}
\inv: &G(\breve{F})/\breve{K}\times G(\breve{F})/\breve{K}\rightarrow \widetilde{W}_{K}\backslash \widetilde{W}/ \widetilde{W}_{K}\\
 &(g,h)\rightarrow \breve{K}g^{-1}h\breve{K}.\\
\end{split}
\end{equation}

Let $b\in G(\breve{F})$ be an element whose image in $B(G)$, the $\sigma$-conjugacy class of $G(\breve{F})$, lies in the subset $B(G, \mu)$ of \emph{neutrally acceptable elements} see \cite[4.5, 4.6]{Rap05}. For $w\in \widetilde{W}_{K}\backslash \widetilde{W}/ \widetilde{W}_{K}$ and $b\in G(\breve{F})$, we define the \emph{affine Deligne-Lusztig variety}
to be the set 
$$X_{w}(b)=\{g\in G(\breve{F})/\breve{K}; \inv(g, b\sigma(g))=w\}.$$
Thanks to the work of \cite{BS-Inv17} and \cite{Zhu-Ann17}, this set can be viewed as an ind-closed-subscheme in the affine flag variety $\breve{G}/\breve{K}$. In the following, we will only consider it as a set. The Rapoport-Zink space is not directly related to the affine Deligne-Lusztig variety but rather to the following union of affine Deligne-Lusztig varieties
$$X(\mu, b)_{K}=\{g\in G(\breve{F})/ \breve{K}; g^{-1}b\sigma(g)\in \breve{K}w\breve{K}, w\in \Adm(\mu) \}.$$
We recall the group $J_{b}$ is defined by the $\sigma$-centralizer of $b$ that is $$J_{b}(R)=\{g\in G(R\otimes_{F}\breve{F}); g^{-1}b\sigma(g)=b\}$$ for any $F$-algebra $R$. In the following we will let $b$ be the unique basic element in $B(G,\mu)$ and in this case $J_{b}$ is an inner form of $G$ see \cite{RR96}. Note that $\tau$ is contained in the basic class $[b]$. 

\subsection{EKOR stratification} We define $\Adm^{K}(\mu)$ to be the image of  $\Adm(\mu)$ in $\widetilde{W}_{K}\backslash\widetilde{W}/\widetilde{W}_{K}$ and $^{K}\widetilde{W}$ to be the set of elements of minimal length in $\widetilde{W}_{K}\backslash\widetilde{W}$.  We define the set $$\mathrm{EKOR}^{K}(\mu)=\Adm^{K}(\mu)\cap ^{K}\widetilde{W}.$$ 
A \emph{$K$-stable piece} is a subset of $G(\breve{F})$ of the form $\breve{K}\cdot_{\sigma}\breve{I}w\breve{I}$ where $\cdot_{\sigma}$ means $\sigma$-conjugation and $\breve{I}$ is an Iwahori subgroup and $w\in {^{K}\widetilde{W}}$. Then we define the \emph{Ekedahl-Kottwitz-Oort-Rapoport stratum}(EKOR stratum) attached to $w\in \mathrm{EO}^{K}(\mu)$ of $X(\mu, b)_{K}$ by the set $$X_{K,w}(b)=\{g\in G(\breve{F})/\breve{K}; g^{-1}b\sigma(g)\in  \breve{K}\cdot_{\sigma}IwI\}.$$  Then by \cite{GH15} we have the following \emph{EKOR stratification}
\begin{equation}X(\mu, b)_{K}=\bigcup_{w\in\mathrm{EKOR}^{K}(\mu)}X_{K,w}(b).\end{equation} 
This is the local analogue of the stratification defined for Shimura varieties in \cite{HR17} which should be thought of as interpolating between the \emph{Ekedahl-Oort stratification} and the \emph{Kottwitz-Rapoport stratification}.

\subsection{Fully Hodge-Newton decomposable ADLV} In \cite{GHN16} the authors introduced the notion of \emph{Fully Hodge-Newton decomposable pair} $(G,\mu)$ where $G$ and $\mu$ are defined as before. We will only state the following equivalent characterization of this notion. 

\begin{theorem}
Let $(G,\mu)$ be a pair as before and let $K\subset \tilde{\mathbb{S}}$ be such that $\sigma(K)=K$ and $W_{K}$ finite. Then the following are equivalent
\begin{enumerate}
\item The pair $(G,\mu)$ is fully Hodge-Newton decomposable.
\item For each $w\in \Adm(\mu)$, there exists a unique $[b^{\prime}]\in B(G,\mu)$ such that $\breve{I}w\breve{I}\subset [b^{\prime}]$.
\item For any $w\in \mathrm{EKOR}^{K}(\mu)$, there exists a unique $[b^{\prime}]\in B(G,\mu)$ such that $\breve{K}._{\sigma}\breve{I}w\breve{I}\subset [b^{\prime}]$. 
\end{enumerate}
\end{theorem}

Intuitively, $(3)$ means that each EKOR stratum lies in a unique Newton stratum. Recall that $[b]\in B(G,\mu)$ is the unique basic class and we have the following characterization of those EKOR strata lying completely in the basic Newton stratum. For $w\in W_{a}$, we let $\mathrm{supp}(w)$ be the support of $w$ and we set 
$$\mathrm{supp}_{\sigma}(w\tau)=\bigcup_{n\in \ZZ}(\tau\sigma)^{n}(\mathrm{supp}(w)).$$

\begin{proposition}[{\cite[Proposition 4.6]{GHN16}}]
Let $x\in \tilde{W}$. The following are equivalent
\begin{enumerate}
\item $\breve{K}._{\sigma}\breve{I}w\breve{I}\subset [b]$;
\item $\kappa_{G}(x)=\kappa_{G}(b)$ and $W_{\mathrm{supp}_{\sigma}(x)}$ is finite.
\end{enumerate}
\end{proposition}

The above theorem motivates us to introduce the following set 
$$\mathrm{EKOR}^{K}(\mu)_{0}=\{w\in \mathrm{EKOR}^{K}(\mu): W_{\mathrm{supp}_{\sigma}(w)} \text{ is finite}\}.$$
The fully Hodge-Newton decomposable pair gives rise to affine Deligne-Lusztig varieties whose EKOR strata have many nice properties. First of all, we have the following result.

\begin{theorem}
Suppose $(G, \mu)$ is a fully Hodge-Newton decomposable pair. Then 
\begin{enumerate}
\item $X(\mu, b)_{K}= \bigsqcup_{w\in \mathrm{EKOR}^{K}(\mu)_{0}} X_{K, w}(b).$
\item $X_{K,w}(b)$ is a finite union of classical Deligne-Lusztig varieties up to perfection.  
\end{enumerate}
\end{theorem}
 
\begin{proof}
These can be deduced from the main result of \cite[Theorem 2.3]{GHN16}.
\end{proof}

\subsection{Quaternionic unitray case} The pair $(G=\GU_{D}(2), \mu)$ chosen in \S\ref{local-shi-datum} gives rise to a fully Hodge-Newton decomposable pair. More precisely, it gives rise to the \emph{Tits datum} over $\QQ_{p}$ of the type $(\tilde{C}_{2}, w^{\vee}_{2}, \tau_{2})$. Here $w^{\vee}_{2}=\mu$ is the miniscule cocharacter we have defined before and $\tau_{2}$ is the image of $w^{\vee}_{2}$ in $\Omega$ and $\sigma$ acts on the absolute local Dynkin diagram of type $\tilde{C}_{2}$ by $\tau_{2}$. This action on the local Dykin diagram is shown below.

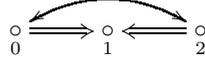
\begin{displaymath}
 \xymatrix{\underset{0}\circ \ar@/^1pc/[rr]\ar@2{->}[r] &\underset{1}\circ&\underset{2}\circ\ar@2{->}[l] \ar@/_1pc/[ll]}
\end{displaymath}

\subsubsection{The paramodular case} When $K_{1}= \tilde{\mathbb{S}}-\{1\}=\{s_{0}, s_{2}\}$, the basic EKOR strata $\mathrm{EKOR}^{K_{1}}(\mu)_{0}$ are tabulated below

\begin{center}
\begin{tabular}{|c c c c|}
\hline
$w_{\emptyset}=\tau$    &$w_{1}=s_{1}\tau$   &$w_{12}=s_{1}s_{2}\tau$  &$w_{10}=s_{1}s_{0}\tau$ \\
\hline                   
\end{tabular}.
\end{center}

We will use similar notations to denote elements in $\tilde{W}$, for example $w_{121}=s_{1}s_{2}s_{1}\tau$. Here the EKOR strata correspond to the Bruhat-Tits strata introduced in \ref{BT-strata}. 
\begin{itemize}
\item The stratum $X_{K_{1}, w_{\emptyset}}(b)$ is the lattice stratum $\calM_{\{1\}}$ which is zero dimensional;
\item The stratum $X_{K_{1}, w_{1}}(b)$ is the lattice stratum $\calM^{\circ}_{\{02\}}$ which is one dimensional;
\item The stratum $X_{K_{1}, w_{12}}(b)$ is the lattice stratum $\calM^{\circ}_{\{0\}}$ which is two dimensional;
\item The stratum $X_{K_{1}, w_{10}}(b)$ is the lattice stratum $\calM^{\circ}_{\{2\}}$ which is two dimensional.
\end{itemize}

\subsubsection{The Iwahori case} When $K_{\emptyset}=\tilde{\mathbb{S}}-\tilde{\mathbb{S}}=\emptyset$, the basic EKOR strata $\mathrm{EKOR}^{K_{\emptyset}}(\mu)_{0}$ are tabulated below.

\begin{center}
\begin{tabular}{|c c c c c c c c c c c|}
\hline
$w_{\emptyset}$ &$w_{0}$ &$w_{1}$ &$w_{2}$ &$w_{01}$ &$w_{02}$ &$w_{10}$ &$w_{12}$ &$w_{21}$ &$w_{010}$ &$w_{212}$\\
\hline              
\end{tabular}
\end{center}

\begin{proposition}\label{Ekor-fiber}
There is a surjective map $\mathrm{EKOR}^{K_{\emptyset}}(\mu)_{0}\rightarrow \mathrm{EKOR}^{K_{1}}(\mu)_{0}$ and we tabulate the fiber of this map below.

\begin{center}
\begin{tabular}{|c c| c|}
\hline
 &$w_{\emptyset}$     &$\{w_{\emptyset}, w_{0}, w_{2}, w_{02}\}$\\
 &$w_{1}$      &$\{w_{1}, w_{01}, w_{21} \}$\\
 &$w_{12}$    & $\{w_{12}, w_{212}\}$\\
 &$w_{10}$    & $\{w_{10}, w_{010}\}$\\
\hline                   
\end{tabular}
\end{center}
\end{proposition}
\begin{proof}
The point is that the image of an EKOR stratum of Iwahori level under natural map will a union of EKOR strata, see \cite[Proposition 6.11]{HR17}. The conclusion of this proposition is that in our case this union consists of a single element. To show this, we need to calculate the set $\Sigma_{K_{1}}(w)$ defined in \cite[Proposition 6.6]{HR17}.  We will rely on the algorithm provided in \cite[Proposition 6.7]{HR17} which we reproduce below 
\begin{enumerate}
\item If $x\in {^{J}\widetilde{W}}$ and $u\in W_{I(J, x, \sigma)}$, then $\Sigma_{K_{1}}(ux)=\Sigma_{K_{1}}(x)$;
\item If $x\in {^{J}\widetilde{W}}$ and $s\in J$ with $l(sw\sigma(s))=l(w)$, then $\Sigma_{K_{1}}(w)=\Sigma_{K_{1}}(sw\sigma(s))$;
\item If $x\in {^{J}\widetilde{W}}$ and $s\in J$ with $l(sw\sigma(s))<l(w)$ then $\Sigma_{K_{1}}(w)=\Sigma_{K_{1}}(sw\sigma(s))\cup \Sigma_{K_{1}}(sw)$.
\end{enumerate}
By \cite[Proposition 6.6]{HR17}, we also know that if $w\in {^{K_{1}}\widetilde{W}}$, then $\Sigma_{K_{1}}(w)=\{w\}$.  Using these results, we can calculate $\Sigma_{K_{1}}(w)$ for each $w\in \mathrm{EKOR}^{K_{\emptyset}}(\mu)_{0}$.  In fact, we have
\begin{itemize}
\item Since $W_{I(J, \tau, \sigma)}=W_{K_{1}}$, $\Sigma_{K_{1}}(\tau)=\Sigma_{K_{1}}(s_{0}\tau)=\Sigma_{K_{1}}(s_{2}\tau)=\Sigma_{K_{1}}(s_{0}s_{2}\tau)=\{\tau\}$ by $(1)$ above.
\item Since $W_{I(J, s_{1}\tau, \sigma)}=W_{K_{1}}$,  $\Sigma_{K_{1}}(s_{1}\tau)=\Sigma_{K_{1}}(s_{0}s_{1}\tau)=\Sigma_{K_{1}}(s_{2}s_{1}\tau)=\{s_{1}\tau\}$ by $(1)$ above.

\item Since $W_{I(J, s_{1}s_{2}\tau, \sigma)}=W_{K_{1}}$,  $\Sigma_{K_{1}}(s_{1}s_{2}\tau)=\Sigma_{K_{1}}(s_{2}s_{1}s_{1}\tau)=\{s_{1}s_{2}\tau\}$ by $(1)$ above.

\item Since $W_{I(J, s_{1}s_{0}\tau, \sigma)}=W_{K_{1}}$,  $\Sigma_{K_{1}}(s_{1}s_{0}\tau)=\Sigma_{K_{1}}(s_{0}s_{1}s_{0}\tau)=\{s_{1}s_{0}\tau\}$ by $(1)$ above.

\end{itemize}
\end{proof}

\begin{remark}
We thank Ulrich G\"{o}rtz for suggesting using \cite[Proposition 6.7]{HR17} for the above proof.
\end{remark}

From the above table one can conclude that $$\mathrm{EKOR}^{K_{\emptyset}}(\mu)_{0}=\overline{\breve{I}w_{212}\breve{I}}\cup \overline{\breve{I}w_{010}\breve{I}}\cup \overline{\breve{I}w_{02}\breve{I}}.$$ It follows that $X(\mu, b)_{K_{\emptyset}}$ has three types of irreducible components, two of them are three dimensional and the other one is one dimensional. Their intersections are all one dimensional and can also be also understood in terms of the above data
\begin{equation*}
\begin{split}
& \overline{\breve{I}w_{212}\breve{I}}\cap \overline{\breve{I}w_{010}\breve{I}}=\breve{I}w_{1}\breve{I};\\
& \overline{\breve{I}w_{010}\breve{I}}\cap \overline{\breve{I}w_{02}\breve{I}}=\breve{I}w_{0}\breve{I};\\
& \overline{\breve{I}w_{212}\breve{I}}\cap \overline{\breve{I}w_{02}\breve{I}}=\breve{I}w_{2}\breve{I}.\\
\end{split}
\end{equation*}
Some of the EKOR strata in this case can also be related to the varieties we defined and studied in the previous sections
 
\begin{itemize}
\item The stratum $X_{K_{\emptyset}, w_{02}}(b)$ can be identified with $\mathcal{Y}_{\{1\}}$ which is two dimensional;
\item The stratum $X_{K_{\emptyset}, w_{21}}(b)$ can be identified with $\mathcal{Y}_{\{02\}}$ which is two dimensional;
\item The stratum $X_{K_{\emptyset}, w_{212}}(b)$ can be identified with $\mathcal{Y}_{\{0\}}$ which is three dimensional;
\item The stratum $X_{K_{\emptyset}, w_{010}}(b)$ can be identified with $\mathcal{Y}_{\{2\}}$ which is three dimensional.
\end{itemize}

\subsubsection{Siegel parahoric case}  When $K_{02}=\tilde{\mathbb{S}}-\{s_{0}, s_{2}\}=\{s_{1}\}$, the basic EKOR strata $\mathrm{EKOR}^{K_{02}}(\mu)_{0}$ are tabulated below.
\begin{center}
\begin{tabular}{|c c c c  |}
\hline
$w_{\emptyset}$ &$w_{0}$ &$w_{2}$ &$w_{02}$ \\
\hline              
\end{tabular}
\end{center}

\begin{proposition}
There is a surjective map $\mathrm{EKOR}^{K_{\emptyset}}(\mu)_{0}\rightarrow \mathrm{EKOR}^{K_{02}}(\mu)_{0}$ and we tabulate the fiber of this map below.

\begin{center}
\begin{tabular}{|c c| c|}
\hline
 &$w_{\emptyset}$     &$\{w_{\emptyset}, w_{1}\}$\\
 &$w_{0}$      &$\{w_{10}, w_{21} \}$\\
 &$w_{2}$    & $\{w_{12}, w_{01}\}$\\
 &$w_{02}$    & $\{w_{010}, w_{212}\}$\\
\hline                   
\end{tabular}
\end{center}
\end{proposition}
\begin{proof}
This follows from the same calculation as in Proposition \ref{Ekor-fiber} . 
\end{proof}

From the proposition, we see that  $\mathrm{EKOR}^{K_{02}}(\mu)_{0}=\overline{\breve{K}_{02}w_{02}\breve{K}_{02}}$ and this agrees with the conclusion in Theorem \ref{main-app} and we have the following identification
\begin{itemize}
\item The closure of the EKOR stratum $X_{K_{02}, w_{02}}(b)$ can be identified with $$\pi_{\{02\}}(\mathcal{Y}_{\{1\}})=\bigsqcup_{x\in \calM_{\{1\}}}(\PP^{1}\times\PP^{1})_{x}.$$ 
\end{itemize}

\section{Applications to Shimura varieties}

\subsection{Integral datum of PEL type} In this section we introduce the integral datum of PEL type that is needed to define the integral model of the quaternionic unitary Shimura variety of parahoric level structure. Let $B$ be an indefinite quaternion algebra over $\QQ$ with neben involution $*$ which fixes a maximal order $\mathcal{O}_{B}$. We assume that $p$ divides the discriminant $\Delta_{B}$. We denote by $V=B\oplus B$ and let $$(\cdot,\cdot): V\times V\rightarrow \QQ$$ be an alternating form such that $(av_{1}, v_{2})=(v_{1}, a^{*}v_{2})$ for all $v_{1}, v_{2}\in V$ and $a\in B$. For a definite choices of these data, see \cite{KR00}. We define the group $G\subset \GL_{B}(V)$ over $\QQ$ to be
$$G(\QQ)=\{g\in \GL_{B}(V): (gv_{1}, gv_{2})=c(g)(v_{1}, v_{2})\text{ for some $c(g)\in \QQ^{\times}$}\}.$$ Since $B$ splits at $\RR$, we in fact have $G(\RR)\cong \GSp(4)(\RR)$. 
Let $\mathbb{S}=Res_{\CC/\RR}\mathbb{G}_{m}$ be the Deligen torus and let $h: \mathbb{S}\rightarrow G_{\RR}$ be the Hodge cocharacter that it induces the miniscule cocharacter $$\mu: \mathbb{G}_{m, \CC}\rightarrow {G}_{\CC}\cong \GSp(4)_{\CC}$$ that sends $z\in \mathbb{G}_{m, \CC}$ to  $\text{diag}(z,z,1,1)\in  \GSp(4)_{\CC}$. 
Recall the periodic lattice chain $\mathcal{L}=\mathcal{L}_{\emptyset}, \mathcal{L}_{\{1\}}, \mathcal{L}_{\{02\}}$ introduced in \eqref{chain-iw}, \eqref{chain-para} and \eqref{chain-sieg}. They give rise to $U_{\mathcal{L}}=\Stab(\mathcal{L})$ which is a parahoric subgroup of $G(\QQ_{p})$. Note that $U_{\mathcal{L}_{\emptyset}}=\Iw$. Finally we fix an open compact subgroup $U^{p}\subset G(\mathbb{A}^{p}_{f})$ that is sufficiently small. The datum $(B, \mathcal{O}_{B}, V, (\cdot, \cdot), \mu, \mathcal{L}, U^{p})$ gives an integral datum of PEL type and is associated to the moduli problem $\Sh_{\mathcal{L}, U^{p}}$ over $\ZZ_{p}$ , compare \cite[Definition 6.9]{RZ96}.

\begin{definition}
A point of the functor $\Sh_{\mathcal{L}, U^{p}}$ with values in a $\ZZ_{p}$-scheme $S$ is given by the following set of data up to isomorphism
\begin{enumerate}
\item For each $\Lambda\in \mathcal{L}$, $A_{\Lambda}$ is an abelian scheme over $S$ of dimension $4$ and an isogeny $$\rho_{\Lambda, \Lambda^{\prime}}: A_{\Lambda}\rightarrow A_{\Lambda^{\prime}}$$ for any another lattice $\Lambda\subset \Lambda^{\prime}$;
\item For each $\Lambda\in \mathcal{L}$, an isomorphism $$\lambda_{\Lambda}:A_{\Lambda}\rightarrow A^{\vee}_{\Lambda^{\perp}};$$
\item For each $\Lambda\in \mathcal{L}$, a map $$\iota_{\Lambda}: \mathcal{O}_{B}\rightarrow \End(A_{\Lambda});$$
\item A $U^{p}$-level structure $$\bar{\eta}: H_{1}(A,\mathbb{A}^{p}_{f})\cong V\otimes \mathbb{A}^{p}_{f} \text{ mod } U^{p}$$ where $H_{1}(A,\mathbb{A}^{p}_{f})=H_{1}(A_{\Lambda},\mathbb{A}^{p}_{f})$ for any $\Lambda\in \mathcal{L}$.
\end{enumerate}
We require that these data satisfy the following conditions
\begin{enumerate}
\item For any geometric point $s\in S$ of characteristic different from $p$,
\begin{equation*}
\begin{split}
&T_{p}(A_{\Lambda, s})\cong \Lambda;\\
&T_{p}(A_{\Lambda^{\prime}, s})/T_{p}(A_{\Lambda, s})\cong\Lambda^{\prime}/\Lambda.\\
\end{split}
\end{equation*}
\item $\lambda^{-1}_{\Lambda}\circ\iota_{\Lambda}(b)^{\vee}\circ\lambda_{\Lambda}=\iota_{\Lambda}(b^{*})$.
\item For each $\Lambda\in\mathcal{L}$, $\iota_{\Lambda}$ satisfies the Kottwitz condition $$\det(T-\iota_{\Lambda}(b); \Lie(A_{\Lambda}))=(T^{2}-\Trd(b)T+\Nrd(b))^{2}$$
for all $b\in B$.
\end{enumerate}
\end{definition}

This functor is representable by a quasi-projective scheme $\Sh_{\mathcal{L}, U^{p}}$ over $\ZZ_{p}$ and its relative dimension is $3$.  For a $\ZZ_{p}$-scheme $S$, an $S$-point of $x\in \Sh_{\mathcal{L}, U^{p}}(S)$ will be represented by a tuple $x=(A_{\Lambda}, \lambda_{\Lambda}, \iota_{\Lambda},\bar{\eta})$.

\subsection{$p$-adic uniformization of the supersingular locus} Let $Sh_{\mathcal{L}, U^{p}}$ be the special fiber of $\Sh_{\mathcal{L}, U^{p}, W_{0}}$ at $p$. We denote by $Sh^{ss}_{\mathcal{L}, U^{p}}$ the supersingular locus of $Sh_{\mathcal{L}, U^{p}}$. This is by definition the reduced $\FF$-scheme whose $\FF$ points are given by the set
$$\{(A_{\Lambda}, \lambda_{\Lambda}, \iota_{\Lambda},\bar{\eta})\in Sh_{\mathcal{L}, U^{p}}(\FF): \text{$A_{\Lambda}$ is supersingular}\}.$$
We fix a $\FF$-point $(A_{0,\Lambda}, \lambda_{0, \Lambda}, \iota_{0,\Lambda}, \bar{\eta}_{0})$ and let $\mathbf{A}$ be an abelian variety isogenous to $A_{0,\Lambda}$ and whose associated $p$-divisible group is isomorphic to $\XX$ which was used to define the Rapoport-Zink space $\mathcal{N}$. We define the algebraic group $I$ over $\QQ$ by $$I(\QQ)=\{g\in \End^{0}(A_{0,\Lambda})^{\times}: g^{\vee}\circ \lambda_{0,\Lambda}\circ g=c(g)\lambda_{0,\Lambda}\text{ for some $c(g)\in\QQ^{\times}$}\}.$$
We note that $I(\QQ_{p})=J(\QQ_{p})$. Let $\widehat{Sh}^{ss}_{\mathcal{L}, U^{p}}$ be the completion of $\Sh_{\mathcal{L}, U^{p}, W_{0}}$ along $Sh^{ss}_{\mathcal{L}, U^{p}}$. Then the following theorem is known as the Rapoport-Zink uniformization theorem
\begin{theorem}[{\cite[Theorem 6.30]{RZ96}}]\label{RZ-unif}
Given $(A_{0,\Lambda}, \lambda_{0, \Lambda}, \iota_{0,\Lambda}, \bar{\eta}_{0})\in Sh_{\mathcal{L}, U^{p}}(\FF)$, there is an isomorphism of formal schemes over $\Spf(W_{0})$
$$I(\QQ)\backslash \mathcal{N}_{\mathcal{L}}\times G(\mathbb{A}^{p}_{f})/ U^{p} \cong \widehat{Sh}^{ss}_{\mathcal{L}, U^{p}}.$$
\end{theorem}

The above theorem allows us to transfer the results about the structure of $\mathcal{N}_{\calL}$ to results about $Sh^{ss}_{\mathcal{L}, U^{p}}$.  To simplify the notation, we set 
\begin{equation}
\begin{split}
&Sh^{ss}_{U^{p}}=Sh^{ss}_{\mathcal{L}_{\{1\}}, U^{p}}\\
&Sh^{ss}_{P,U^{p}}=Sh^{ss}_{\mathcal{L}_{\{0,2\}}, U^{p}}\\ 
&Sh^{ss}_{\Iw, U^{p}}=Sh^{ss}_{\mathcal{L}_{\{\emptyset\}}, U^{p}}.\\
\end{split}
\end{equation}

\begin{theorem}\label{main-shim}\hfill
\begin{enumerate}
\item The scheme $Sh^{ss}_{U^{p}}$ is purely $2$-dimensional and an irreducible component of it is a surface of the form $x^{p}_{3}x_{0}-x^{p}_{0}x_{3}+x^{p}_{2}x_{1}-x^{p}_{1}x_{2}=0$.

\item The scheme $Sh^{ss}_{P, U^{p}}$ is purely $2$-dimensional and an irreducible component of it is the ruled surface $\PP^{1}\times \PP^{1}$.

\item The scheme $Sh^{ss}_{\Iw, U^{p}}$ has both $2$-dimensional components and $3$-dimensional components. The $2$-dimensional components are of the form $\PP^{1}\times \PP^{1}$ and a $3$-dimensional component is a $\PP^{1}$-bundle over the surface $x^{p}_{3}x_{0}-x^{p}_{0}x_{3}+x^{p}_{2}x_{1}-x^{p}_{1}x_{2}=0$. In particular since the dimension of $Sh_{\Iw, U^{p}}$ is $3$, there are irreducible components of it that are purely supersingular and thus the ordinary locus of $Sh_{\Iw, U^{p}}$ is not dense. 

\end{enumerate}
\end{theorem}
\begin{proof}
This follows directly using the Rapoport-Zink uniformization theorem Theorem \ref{RZ-unif} and Theorem \ref{main-thm} and Theorem \ref{main-app}.
\end{proof}

Recall that there are two types of irreducible components of $Sh^{ss}_{U^{p}}$ corresponding to vertex lattices of type $0$ and $2$ respectively, we will refer to them as type $0$ components and type $2$ components. By Theorem \ref{main-shim}, the irreducible components of $Sh^{ss}_{\Iw, U^{p}}$ of dimension $3$ are $\PP^{1}$-bundles over the irreducible components of $Sh^{ss}_{U^{p}}$ and therefore are naturally divided into two types which we still refer to as type $0$ and $2$ components. We also see that there are additional two dimensional components of $Sh^{ss}_{\Iw, U^{p}}$ which we will refer to as type $1$ component.  Let $P_{\{i\}}$ be the stabilizer in $J_{b}(\QQ_{p})$ of a fixed vertex lattice of type $i$ for $i=0, 1, 2$. We easily obtain the following parametrization of irreducible components of $Sh^{ss}_{\mathcal{L},U^{p}}$. We remark that those double cosets appear below can be regarded as discrete Shimura varieties or Shimura set and should have arithmetic applications related to geometrically realizing Jacquet-Langlands correspondence for symplectic groups of degree $4$. 

\begin{proposition}
For the paramodular parahoric case 
\begin{itemize}
\item There is an one-to-one correspondence 
$$\{\text{irreducible components of $Sh^{ss}_{ U^{p}}$ of type $0$}\}\longleftrightarrow I(\QQ)\backslash J_{b}(\QQ_{p})/P_{\{0\}}\times G(\mathbb{A}^{p}_{f})/U^{p}.$$
\item There is an one-to-one correspondence 
$$\{\text{irreducible components of $Sh^{ss}_{U^{p}}$ of type $2$}\}\longleftrightarrow I(\QQ)\backslash J_{b}(\QQ_{p})/P_{\{2\}}\times G(\mathbb{A}^{p}_{f})/U^{p}.$$
\end{itemize}

For the Siegel parahoric case
\begin{itemize}
\item There is an one-to-one correspondence 
$$\{\text{irreducible components of $Sh^{ss}_{P, U^{p}}$}\}\longleftrightarrow I(\QQ)\backslash J_{b}(\QQ_{p})/P_{\{1\}}\times G(\mathbb{A}^{p}_{f})/U^{p}.$$
\end{itemize}
For the Iwahori case
\begin{itemize}
\item There is an one-to-one correspondence 
$$\{\text{irreducible components of $Sh^{ss}_{\Iw, U^{p}}$ of type $0$}\}\longleftrightarrow I(\QQ)\backslash J_{b}(\QQ_{p})/P_{\{0\}}\times G(\mathbb{A}^{p}_{f})/U^{p}.$$
\item There is an one-to-one correspondence 
$$\{\text{irreducible components of $Sh^{ss}_{\Iw, U^{p}}$ of type $2$}\}\longleftrightarrow I(\QQ)\backslash J_{b}(\QQ_{p})/P_{\{2\}}\times G(\mathbb{A}^{p}_{f})/U^{p}.$$
\item There is an one-to-one correspondence 
$$\{\text{irreducible components of $Sh^{ss}_{\Iw, U^{p}}$ of type $1$}\}\longleftrightarrow I(\QQ)\backslash J_{b}(\QQ_{p})/P_{\{1\}}\times G(\mathbb{A}^{p}_{f})/U^{p}.$$
\end{itemize}
\end{proposition}
\begin{proof}
This follows immediately the Rapoport-Zink uniformization theorem Theorem \ref{RZ-unif} and Theorem \ref{main-thm} and Theorem \ref{main-app}.
\end{proof}

\end{document}